\documentclass[10pt,twoside]{article}
\usepackage{graphics,color}
\usepackage{graphicx}
\usepackage{epstopdf}
\usepackage{amsfonts,amsmath,amsthm,amssymb,bm}
\usepackage{ulem}
\usepackage{xcolor}
\usepackage{url}

\newcommand{\ds}{\displaystyle}

\newtheorem{thm}{Theorem}[section]
\newtheorem{prop}[thm]{Proposition}
\newtheorem{lem}[thm]{Lemma}

\newtheorem{exmp}{Example}[section]
\newtheorem{remark}{Remark}[section]

\begin{document}

\title{Stability and bifurcations in  Wilson-Cowan systems with distributed delays, and an application to basal ganglia interactions}

\author{Eva Kaslik$^{1}$ \and Emanuel-Attila Kokovics$^1$ \and
Anca R\u adulescu$^{2,*}$
}

\date{\noindent $^1$ West University of Timi\c{s}oara, Bd.  V. P\^{a}rvan nr. 4, 300223, Timi\c{s}oara, Romania\\
\noindent $^2$ State University of New York at New Paltz, New Paltz, NY 12561, USA\\ 
$^*$ Corresponding Author: {\it radulesa@newpaltz.edu}}

\maketitle

\begin{abstract}

The traditional Wilson-Cowan model of excitatory and inhibitory meanfield interactions in neuronal populations considers a weak Gamma distribution of time delays when processing inputs, and is obtained via a time-coarse graining technique that averages the population response. { Previous analyses of the stability of the Wilson-Cowan model focused on more simplified cases, where the delays were either not present, constant or were of a specific type}.  Since these simplifications may significantly alter the behavior of the model, we focus on understanding the behavior of the system before time-course graining, and for a wider range of delay distributions.

For these generalized delay equations, we perform stability and bifurcation analyses with respect to parameters that capture both the coupling profile, and the time delay. { The investigation is done through the examination of the system's associated characteristic equation. Under mild assumptions, we give complete mathematical proofs of our theoretical results, for the model with general delay distributions and prove the transversality condition for the possible Hopf bifurcations, in a generalized context}. The stability region in this parameter space is described theoretically for several types of delay kernels, and numerical simulations are presented to substantiate the theoretical results. 
   
We found that the stability regions and bifurcations differ significantly between different types of delay distributions: weak Gamma distributions promote stable firing rates, while strong Gamma distributions are associated with regular { oscillations}, and Dirac distributions appear to facilitate more complex {, aperiodic} patterns. This supports the unexplored possibility of different delay distributions being used as the substrate for different functional behaviors, and emphasizes the importance of a careful choice of the delay kernel in the mathematical model.
    
We illustrate these theoretical principles in an application to a basal ganglia circuit, in which $\beta$-band { oscillations} have been associated with Parkinson's Disease.

\end{abstract}

\section{Introduction}

\subsection{Modeling background}

Computational modeling of neuronal behavior covers a large range of spatio-temporal scales, from detailed, membrane potential based models of single spiking neurons, to broad network models of interacting brain regions. At the lower scale, typically associated with electrode recordings from single cell in vitro preparations, modeling difficulties often arise from the inherent complexity and high dimensionality associated with considering detailed molecular mechanisms that govern ionic currents and spiking activity in a single cell. The Hodgkin-Huxley model was in its original form limited to the two voltage-dependent currents found in the squid giant axon, but it has to be extended to dozens of equations per neuron if it includes other ion channels involved in neuronal excitability~\cite{meunier2002playing,rinzel1990discussion}. In the context of studying behavior in functional neuronal networks, which may involve thousands of neurons, the dimensionality of a network formed of single cells may become an obstacle to computational feasibility. At the opposite end, often associated with functional imaging data, modeling activity within entire brain regions as a whole lacks specificity, and prevents clear interpretations of what ``activity'' of a state variable may actually represent~\cite{vertes2012simple}.

The middle range between the two ends consists of a rich variety of models. One wide-spread possibility has been creating reduced models as modifications of Hodgkin-Huxley equations, by modeling the behavior of multiple ion channels into one comprehensive variable~\cite{golomb1994clustering}. Another practical option has been using state variables to characterize the meanfield spiking activity in a population of cells~\cite{bick2020}. This type of model is still able to incorporate information on spiking mechanisms, and efficiently illustrates resulting firing patterns by using only one variable per population.

One of the most historically significant and best understood meanfield models, the Wilson-Cowan model \cite{Wilson-Cowan} is perhaps the most popular. The model, derived in 1972, describes the localized interactions in a pair of excitatory and inhibitory neuronal populations. At each time instant $t$, the proportions of excitatory and inhibitory cells firing per unit of time are captured by the two state variables $E(t)$ and $I(t)$. The original model considers the effect that an external input $P$ has on the E/I system, based not only on the coupling strengths between the two units, but also on the history of firing in each. More specifically,  $E(t+{\tau_1})$ and $I(t+{\tau_2})$  represent the proportion of cells which are sensitive (i.e. not refractory) and which also receive at least threshold excitation at the moment of time $t$. This leads to the following system of integral equations:
\begin{equation}
\left\{\begin{array}{l}
\ds E(t+{\tau_1})=\left(1-\int_{t-r}^tE(s)ds\right)\cdot\mathcal{S}_e\left[\int_{-\infty}^t h(t-s)\left(c_1E(s)-c_2I(s)+P_e(s)\right)ds\right]\\
\ds I(t+{\tau_2})=\left(1-\int_{t-r'}^tI(s)ds\right)\cdot\mathcal{S}_i\left[\int_{-\infty}^t h(t-s)\left(c_3E(s)-c_4I(s)+P_i(s)\right)ds\right]
\end{array}\right.
\label{unsimplified}
\end{equation}
Here, the first factor on each right hand side represents the proportion of sensitive excitatory / inhibitory cells, where $r$ and $r'$ are the absolute refractory periods (msc).  The functions $\mathcal{S}_e$, $\mathcal{S}_i$ are sigmoid threshold functions, their arguments denoting the mean field level of excitation / inhibition generated in an excitatory /inhibitory cell at time $t$.  The summation coefficients $c_i>0$ are connectivity weights representing the average number of excitatory / inhibitory synapses per cell and $P_e,P_i$ denote external inputs. It is additionally assumed that  a cell's weighted summation of inputs is modulated longitudinally by a decreasing kernel $h(t)$. The original paper \cite{Wilson-Cowan} suggests the use of a weak Gamma kernel $h(t)$, before proceeding with further mathematical simplifications, for convenience of the analysis. 

 Indeed, with the additional simplifying assumption of negligible refractory periods ($r=r'=0$) and with a degree one Taylor approximation of the left-hand terms for small ${\tau_1}$ and ${\tau_2}$, the Equations~\eqref{unsimplified} can be easily rewritten as the following system of intergo-differential equations:
\begin{equation}
\left\{\begin{array}{l}
\ds \dot{E}(t)=-{\tau_1} E(t) + \mathcal{S}_e\left[\int_{-\infty}^t h(t-s)\left(c_1E(s)-c_2I(s)+P_e(s)\right)ds\right]\\
\ds \dot{I}(t)=-{\tau_2} I(t) +\mathcal{S}_i\left[\int_{-\infty}^t h(t-s)\left(c_3E(s)-c_4I(s)+P_i(s)\right)ds\right]
\end{array}\right.
\label{simplified1}
\end{equation}

For simplicity, the equations are often considered to have the same time scale ${\tau_1=\tau_2=1}$. The model historically 
used as the Wilson-Cowan model \cite{Wilson-Cowan} was then obtained  from~\eqref{simplified1} by applying further time-coarse graining. This final form, consisting of a system of non-delayed ordinary differential equations, is very convenient and was extensively analyzed and used in the modeling literature~\cite{destexhe2009wilson}. However,  this final course-graining step obscures potentially vital information { on the the shape of the temporal integration of synaptic inputs in both populations (by assuming, for example, that $h(t)$ is close to one for $t<r,r'$ respectively, and that is decays rapidly to zero for $t>r,r'$)}. This information may be crucial to the neural function that the model is aiming to address, { hence a few generalizations of the standard model found it useful to further explore the original equations (prior to coarse graining), imposing loser conditions on temporal integration, such as discrete time-delays, and often considering null refractory periods $r,r'$~\cite{Coombes-2009,visser2012analysis,pasillas2013delay,veltz2013interplay}.}

However, constant delays do not constitute the most biologically realistic possibility. Since $h(t-s)$ represents the effect of presynaptic input at time s on the membrane potential at time t, it can be thought to include two aspects: the synaptic transmission delay, and the time-course of the post-synaptic potential (PSP). In this context, constant delays imply not only equal transmission delays between the neurons, but also instantaneous PSPs. While often delays due to PSP are ignored, and a constant kernel is chosen on the basis on experimentally observed synaptic transmission delays~\cite{holgado2010}, such choices are rather made in the interest of simplicity, and do not incorporate the full picture, and other types of distributed delays may deliver a more realistic model of synaptic transmission delays.

\color{black}

 In this paper, we consider and analyze the Wilson-Cowan model in its original integro-differential form, { for a general delay distribution and { with the only simplifying assumption (retained for simplicity) being that of zero refractory periods $r=r'=0$}. We then compare the results} for several types of delay kernels. We analyze the long-term behavior in each case, then we discuss the main differences between these behaviors, and potential advantages of using delay equations versus the simpler and more main-stream course grain approximation.
 
 We emphasize, that our main theorems have already been briefly presented (without proofs) in \cite{kaslik2020wilson}, accompanied by several numerical simulations, mostly in the case of a discrete time-delay. Thus, in this paper we illustrate rigorous mathematical proofs of our previously presented statements \cite{kaslik2020wilson} and of our additional lemmas and propositions added here. Choosing different parameters similar numerical simulations are presented, as in \cite{kaslik2020wilson,Coombes-2009}. Finally a Parkinson's Disease (PD) model of the basal ganglia is adapted after \cite{holgado2010}, followed by a series of simulations, treating both the case of healthy basal ganglia and a diseased/parkinsonian one, using different delay distributions. We then present detailed neuro-biological interpretations and argumentations of our simulations of the PD model.

\subsection{Our model}

To address the integral terms from the original equations ~\eqref{simplified1}  we analyze in this paper the general model with distributed delays  defined by
\begin{equation}\label{sys.wilson.cowan.dd}
\left\{\begin{array}{l}
\ds \dot{u}(t)=-u(t)+f_1\left[\theta_u+\int\limits_{-\infty}^t h(t-s)\left(au(s)+bv(s)\right)ds\right]\\
\ds \dot{v}(t)=-v(t)+f_2\left[\theta_v+\int\limits_{-\infty}^t h(t-s)\left(cu(s)+dv(s)\right)ds\right]\\
\end{array}\right.
\end{equation}
where $u(t)$ and $v(t)$ represent the firing activity in the two neuronal populations, $a,b,c,d$ are the connection weights and $\theta_u$, $\theta_v$ are background drives. The activation functions $f_1$ and $f_2$  are smooth and increasing on the real line.  The original model~\eqref{simplified1} can be recovered for equal time scales ${\tau_1=\tau_2}$, if the activation functions are particularly set to sigmoidals $f_1 = \mathcal{S}_e$ and $f_2 = \mathcal{S}_i$, with background drives 
$\ds \theta_u = \int_{-\infty}^t h(t-s) P_e(s) ds$ and $\ds \theta_v = \int_{-\infty}^t h(t-s) P_i(s) ds$, respectively.

The delay kernel  $h:[0,\infty)\to[0,\infty)$ in system~\eqref{sys.wilson.cowan.dd} is suggested to have an exponential form (weak Gamma kernel) in the original Wilson-Cowan reference~\cite{Wilson-Cowan}. We will consider the extensive case of a general probability density function, representing the probability that a particular time delay occurs. The delay kernel is considered to be bounded and piecewise continuous, satisfying
\begin{equation}\label{delay.kernel.properties}
\int\limits_0^{\infty}h(s)ds=1,\quad \textrm{with the average time delay} \quad\tau=\int\limits_0^{\infty}sh(s)ds<\infty.
\end{equation}

Our analysis of the system~\eqref{sys.wilson.cowan.dd} will focus around understanding the changes in its long term behavior that are triggered by changes in spatial and temporal summation of inputs (as reflected in the synaptic weights and in the integration kernel).

In recent years, significant work has been dedicated to understanding the effects of spatial integration of inputs~\cite{kaiser2007brain,curto2017can}. An entire direction in computational neuroscience is centered around investigating the effects of connectivity and synaptic weights on dynamics in coupled neural populations~\cite{vyas2020computation}. Seminal empirical work has shown that the synaptic weight profile in a network changes with location~\cite{iyer2013influence}, with task-related circumstances (e.g., when learning and memory formation are involved~\cite{stuchlik2014dynamic}), but also with general circumstances (e.g., sleep and circadian rhythms~\cite{golden2020sleep,parekh2016circadian}). such modeling work is shedding increasingly more light on the mechanisms by which different synaptic landscapes modulate different outcomes in networks of neuronal populations.

Comparatively, less is known about temporal profiles in the integration of inputs. While physiological evidence supports that exponential decay may represent an appropriate model for certain type of neural contexts~\cite{rahman2018exponentially}, other modulation patterns have been explored~\cite{meyer2017models,chow2020before}. It is generally agreed that effects of stimulation decay with the time course, but it is quite possible that the timeline and profile of this decay { (the shape of the post-synaptic integration) depend on the properties of the neurons involved (receptor types, voltage-dependent currents in the soma and dendrites)~\cite{byrne2014postsynaptic}.}

The Wilson-Cowan type system~\eqref{sys.wilson.cowan.dd} simplifies the behavior of a self-interacting E/I network into a mean-field, two-dimensional model without the geometric, spatial structure described in other, more detailed models. The distribution of synaptic weights contributing to spatial summation of inputs is then entirely encapsulated in the set of mean field coefficients $a,b,c,d$. The temporal summation of inputs is described by the kernel $h$. In this paper, we study, for different types of kernels, the types of asymptotic regimes accessible to the system. For each kernel type, we further study the conditions on the synaptic coefficients and distributed delay that produce transitions between these regimes. 

The specific case of discrete time delays (Dirac kernels) has been previously analyzed in
{\cite{Coombes-2009, pavlides2012}}. However, { as previously mentioned,} the discrete time integration model does not have a strong physiological backup; other important types of delay kernels are frequently used in the literature, such as uniform distribution kernels { \cite{Campbell_2009, Jessop_2010,Bernard_2001}} or Gamma kernels
(recall that the original Wilson-Cowan model  itself proposed a weak Gamma { (exponential)} kernel $h(t)=\tau^{-1}\exp(-t/\tau)$) \cite{Wilson-Cowan}. Analyzing  and comparing mathematical models which incorporate  different types of delay kernels (e.g. weak Gamma kernel or strong Gamma kernel $h(t)=4\tau^{-2}t\exp(-2t/\tau)$) may shed a light on the difference between using different distributed delays, as well as on the differences between using continuous versus discrete delays on the system's dynamics. Furthermore, when modeling natural systems, one usually does not have access to the exact delay distribution, and approaches using general kernels may prove to be more revealing \cite{Adimy_2005,Bernard_2001,Campbell_2009,Diekmann_2012,Faria_2008,Jessop_2010,Ozbay_2008,Ruan_1996,Yuan_2011}. 

The results in Section~\ref{main_results} are delivered within this general context, keeping additional conditions on the kernels to a minimum. In Sections~\ref{application1} and~\ref{application2}, we further illustrate the behaviors for specific theoretical and biologically driven examples, using particular kernels (the weak and strong Gamma kernels and the Dirac kernel).  In Section~\ref{discussion}, we draw comparisons between the behaviors reported for different types of distributed delays, and discuss potential contributions of the delay scheme to dynamic control mechanisms in coupled neuronal populations. 

\begin{figure}[http]
\centering
\begin{minipage}[c]{0.48\linewidth}
	\centering
	\includegraphics[width=\linewidth]{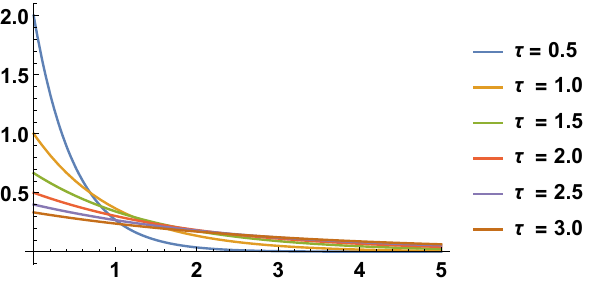}
\end{minipage}
\hspace*{0.02\linewidth}
\begin{minipage}[c]{0.48\linewidth}
\centering
  	\includegraphics[width=\linewidth]{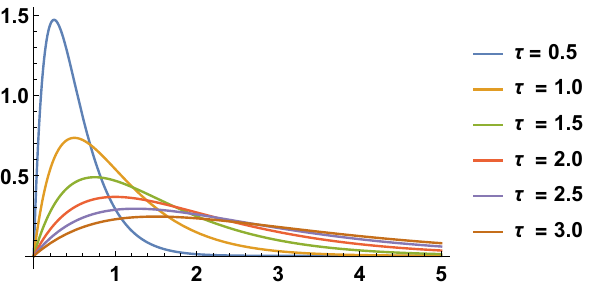}
\end{minipage}
    \caption{Weak and Strong Gamma kernels for different values of $\tau$.}
    \label{fig:gamma1.2}
\end{figure}


\section{Stability and bifurcation results}
\label{main_results}

The initial conditions that are associated to system (\ref{sys.wilson.cowan.dd}) are:
$$u(s)=\psi_1(s),\quad v(s)=\psi_2(s),\quad \forall\, s\in(-\infty,0],$$
where $\psi_1$ and $\psi_2$  belong to the Banach space $C_{0,\mu}(\mathbb{R}_-,\mathbb{R})$ (where $\mu>0$) of continuous real valued functions defined on $(-\infty,0]$ such that $\lim\limits_{t\rightarrow -\infty} e^{\mu t}\psi(t)=0$, endowed with the norm:
$$\|\psi\|_{\infty,\mu}=\sup_{t\in(-\infty,0]}e^{\mu t}|\psi(t)|.$$
For existence and uniqueness results regarding the solution of the initial value problem associated to system \eqref{sys.wilson.cowan.dd} we cite \cite{hale1991introduction,Kolmanovskii1999}. 
{ An important tool in the stability analysis of system \eqref{sys.wilson.cowan.dd} is the \textit{principle of linearized stability} (see for example \cite{smith2011introduction}), characterizing the local stability properties of an equilibrium state of \eqref{sys.wilson.cowan.dd} in terms of the stability of the null solution of the linearized system at that particular equilibrium. For the principle of linearized stability in the context of systems of differential equations with infinite time delay, we refer to \cite{Diekmann_2012}.}

The equilibrium states of system (\ref{sys.wilson.cowan.dd}) are found by solving the following algebraic system:
\begin{equation}\label{sys.ss}
\left\{\begin{array}{l}
\ds u=f_1(\theta_u+au +bv)\\
\ds v=f_2(\theta_v+cu +dv)\\
\end{array}\right.
\end{equation}

Linearizing at an equilibrium state $(u^\star, v^\star)$ leads to
\begin{equation}\label{sys.lin}
\left\{\begin{array}{l}
\ds \dot{u}=-u+\phi_1\int\limits_{-\infty}^t h(t-s)\left(au(s)+bv(s)\right)ds\\
\ds \dot{v}=-v+\phi_2\int\limits_{-\infty}^t h(t-s)\left(cu(s)+dv(s)\right)ds\\
\end{array}\right.
\end{equation}
where $\phi_1=\phi_1(u^\star, v^\star)=f_1'(\theta_u+au^\star+bv^\star)>0$ and $\phi_2=\phi_2(u^\star, v^\star)=f_2'(\theta_v+cu^\star+dv^\star)>0$.

{ Either looking for solutions of the form $(u(t),v(t))=e^{zt}(u(0),v(0))$ or directly} employing the Laplace transform technique to the linearized system (\ref{sys.lin}), we obtain:
\begin{equation}\label{sys.laplace}
\left\{\begin{array}{l}
\ds zU(z)-u(0)=-U(z)+\phi_1H(z)\left(aU(z)+bV(z)\right)\\
\ds zV(z)-v(0)=-V(z)+\phi_2H(z)\left(cU(z)+dV(z)\right)\\
\end{array}\right.
\end{equation}
where $U(z)$, $V(z)$ and $H(z)$ denote the Laplace transforms of state variables $u$ and $v$, and of the delay kernel $h$, respectively.

System (\ref{sys.laplace}) is equivalently written as:
\begin{equation}\label{sys.laplace.matrix}
  \left(
    \begin{array}{cc}
      z+1-a\phi_1H(z) & -b\phi_1 H(z) \\
      -c\phi_2H(z) & z+1-d\phi_2H(z) \\
    \end{array}
  \right)\left(
           \begin{array}{c}
             U(z) \\
             V(z) \\
           \end{array}
         \right)=\left(
                   \begin{array}{c}
                     u(0) \\
                     v(0) \\
                   \end{array}
                 \right)
\end{equation}
which leads to the following characteristic equation (associated to the equilibrium $(u^\star, v^\star)$):
\begin{equation}\label{eq.char}
  \Delta(z)=(z+1)^2-\alpha H(z)(z+1)+\beta H^2(z)=0
\end{equation}
where
\begin{align*}
\alpha&=a\phi_1(u^\star, v^\star)+d \phi_2(u^\star, v^\star)=af_1'(\theta_u+au^\star+bv^\star)+df_2'(\theta_v+cu^\star+dv^\star);\\
\beta&=(ad-bc)\phi_1(u^\star, v^\star)\phi_2(u^\star, v^\star)=(ad-bc)f_1'(\theta_u+au^\star+bv^\star)f_2'(\theta_v+cu^\star+dv^\star).
\end{align*}

\subsection{Delay-independent stability and instability results}

We first discuss several delay-independent stability and instability properties of system (\ref{sys.wilson.cowan.dd}), emphasizing that these results are particularly useful when the exact delay distributions are unknown. 

\begin{thm}\label{thm.stab} Assume that the delay kernel $h(t)$ in system \eqref{sys.wilson.cowan.dd} satisfies the properties \eqref{delay.kernel.properties}.
\begin{enumerate}
\item In the non-delayed case (i.e. $h(t)=\delta(t)$), the equilibrium $(u^\star, v^\star)$ of system \eqref{sys.wilson.cowan.dd} is locally asymptotically stable if and only if the following inequality holds:
\begin{equation}\label{cond.as.stab.no.delay}
\alpha<\min\{2,\beta+1\}
\end{equation}
\item If the following inequality holds:
\begin{equation}\label{cond.as.stab.indep.of.delay}
|\alpha|+|\beta|<1
\end{equation}
then the equilibrium $(u^\star, v^\star)$ of system (\ref{sys.wilson.cowan.dd}) is locally asymptotically stable, regardless of the choice of the delay kernel $h(t)$.
\item If the following inequality holds:
\begin{equation}\label{cond.instab.indep.of.delay}
\beta<\alpha-1
\end{equation}
then the equilibrium state $(u^\star, v^\star)$ of system (\ref{sys.wilson.cowan.dd}) is unstable for any delay kernel $h(t)$.
\end{enumerate}
\end{thm}

\begin{proof}
1. In the non-delayed case ($H(z)=1$, for any $z\in\mathbb{C}$), the characteristic equation (\ref{eq.char}) becomes:
$$\Delta(z)=z^2+(2-\alpha)z+\beta-\alpha+1=0$$
The necessary and sufficient condition (\ref{cond.as.stab.no.delay}) for the asymptotic stability of the equilibrium state $(u^\star, v^\star)$ of system (\ref{sys.wilson.cowan.dd}) follows from the Routh-Hurwitz stability test.

2. { Using basic inequality techniques, we have:
$$|H(z)|=\left|\int\limits_{0}^\infty e^{ -zt}h(t)dt\right|\leq\int\limits_0^\infty |e^{ -z t}|h(t)dt=\int\limits_0^\infty e^{-\Re(z) t}h(t)dt.$$
Now, if we assume that the characteristic equation (\ref{eq.char}) has a root $z$ in the right half-plane ($\Re(z)\geq 0$), it follows that $e^{-\Re(z) t}\leq 1$, for any $t\geq 0$ and hence:
$$
|H(z)|\leq \int\limits_0^\infty h(t)dt=1.
$$
}
From (\ref{eq.char}) we deduce:
\begin{align*}
|z+1|^2&=|\alpha H(z)(z+1)-\beta H^2(z)|\\
&\leq |\alpha||z+1|+|\beta|
\end{align*}
Considering the polynomial $P(x)=x^2-|\alpha|x-|\beta|$, from inequality (\ref{cond.as.stab.indep.of.delay}) it follows that $P(1)>0$ and $P'(1)>0$, and hence $P(x)>0$ for any $x\geq 1$. From the above inequality, we have $P(|z+1|)\leq 0$, and hence, we deduce $|z+1|<1$, which is absurd, since $|z+1|^2=|z|^2+2\Re(z)+1\geq 1$.

Therefore, all the roots of the characteristic equation (\ref{eq.char}) are in the left half-plane and the equilibrium state $(u^\star, v^\star)$ of system (\ref{sys.wilson.cowan.dd}) is asymptotically stable, regardless of the delay kernel $h(t)$.

3. { 
Using the delay kernel properties \eqref{delay.kernel.properties} and the Laplace transform definition, we have $$H(0)=\ds\int\limits_0^\infty h(t)dt = 1,\quad\text{for any }t\geq 0.$$
Based on (\ref{eq.char}), it follows that
\begin{equation}\label{eq.delta.zero}
\Delta(0) = 1 - \alpha H(0) + \beta H^2(0) = 1 - \alpha + \beta.   
\end{equation}
Thus, condition (\ref{cond.instab.indep.of.delay}), i.e. $\beta < \alpha -1$ is equivalent to $\Delta(0)<0$. Given that the characteristic function $\Delta(z)$ is continuous on $[0,\infty)$ and}  $\Delta(z)\rightarrow\infty$ as $z\rightarrow\infty$, the characteristic equation (\ref{eq.char}) has at least one positive real root. Hence, the equilibrium state $(u^\star, v^\star)$ of system (\ref{sys.wilson.cowan.dd}) is unstable, regardless of the delay kernel $h(t)$.
\end{proof}

\subsection{Saddle-node bifurcation}

\begin{thm}[Saddle-node bifurcation]\label{thm.saddle.node}
A saddle-node bifurcation takes place at the equilibrium state $(u^\star, v^\star)$ of system (\ref{sys.wilson.cowan.dd}), regardless of the delay kernel $h(t)$, if and only if $\alpha\neq 2$ and
\begin{equation}\label{cond.saddle.node}
\beta=\alpha-1
\end{equation}
\end{thm}

\begin{proof}
{ Via relationship \eqref{eq.delta.zero}, 
one can see that,} condition (\ref{cond.saddle.node}) is equivalent to $\Delta(0)=0$ (i.e. the characteristic equation has a zero root). Moreover, $z=0$ is a simple root of the characteristic equation if and only if $\alpha\neq 2$.

Let us denote by $z(\beta)$ the root of the characteristic equation \eqref{eq.char} which satisfies $z(\alpha-1)=0$. Taking the derivative with respect to $\beta$ in equation \eqref{eq.char}, we obtain: 
$$2(z+1)\frac{dz}{d\beta}-\alpha H'(z)(z+1)\frac{dz}{d\beta}-\alpha H(z)\frac{dz}{d\beta}+2\beta H'(z)H(z)\frac{dz}{d\beta}+H^2(z)=0,$$
and hence: 
$$\frac{dz}{d\beta}=-\dfrac{H^2(z)}{2(z+1)-\alpha H'(z)(z+1)-\alpha H(z)+2\beta H'(z)H(z)}.$$
As $H(0)=1$ and $H'(0)=-\tau$, it follows that:
$$\left.\frac{dz}{d\beta}\right|_{\beta=\alpha-1}=\dfrac{1}{(\alpha-2)(\tau+1)}\neq 0.$$
{ This completes the proof of transversality. While we have not verified the conditions for nondegeneracy (which are rather difficult to check analytically), one would expect the nonlinearity (hence the saddle-node bifurcation) to be generic.}
\end{proof}

\subsection{Hopf bifurcation}

In the following, we will show that the average time delay $\tau$ of the delay kernel $h(t)$ plays an important role in the Hopf bifurcation analysis.

Let $\hat{h}(t)=\tau h(\tau t)$, for any $t\geq 0$. The function $\hat{h}$ is a probability density function with the mean value:
$$\int\limits_0^{\infty}t\hat{h}(t)dt=\int\limits_0^{\infty}t\tau h(\tau t)dt=\tau^{-1}\int\limits_0^{\infty}uh(u)du=\tau^{-1}\tau=1.$$
The Laplace transform $\hat{H}(z)$ of $\hat{h}(t)$ is
$$\hat{H}(z)=\int\limits_0^\infty e^{  -zt}\hat{h}(t)dt=\int\limits_0^\infty e^{-zt}\tau h(\tau t)dt=\int\limits_0^\infty e^{ -z\frac{u}{\tau}}h(u)du=H\left(\frac{z}{\tau}\right)$$
We assume from now on that $\hat{H}(z)$ is independent of the average time delay $\tau$. In fact, we note that this is true for several classes of delay kernels, such as:
\begin{itemize}
\item for the Dirac kernel (discrete time delay): $\ds\hat{H}(z)=e^{-z}$;
\item for the $p$-Gamma kernel: $\ds\hat{H}(z)=\left(\frac{p}{p+z}\right)^p$.
\end{itemize}

Based on \cite{ilisnskii1976zeros}, we write $\hat{H}$ in the polar form as:
$$
\hat{H}(i\omega)=\rho(\omega)e^{\ds -i\theta(\omega)}
$$
with $\rho(0)=1$, $\theta(0)=0$. 
We further assume that: 
\begin{align*}
&(A_1):\quad \rho \text{ and } \theta \text{ are of class $C^1$ on $\mathbb{R}$};\\
&(A_2): \quad \rho(\omega)>0,  \text{ for any }\omega>0;\\
&(A_3):\quad \rho'(\omega)\leq 0, \text{ for any }\omega>0;\\
&(A_4):\quad \theta'(\omega)>0, \text{ for any }\omega>0.
\end{align*}
It is easy to verify that for the above mentioned probability densities, these assumptions are satisfied. { It is also important to note that $\omega\mapsto \hat{H}(-i\omega)$ is in fact the characteristic function of a random variable with the mean equal to $1$.}

With the change of variable $z\mapsto \ds\frac{z}{\tau}$, the characteristic equation (\ref{eq.char}) becomes
\begin{equation}\label{eq.char.hat}
\hat{\Delta}(z)=\tau^2\Delta\left(\frac{z}{\tau}\right)=(z+\tau)^2-\alpha\tau \hat{H}(z)(z+\tau)+\beta\tau^2 \hat{H}^2(z)=0
\end{equation}
Denoting $\ds Q_\tau(z)=\frac{z+\tau}{\tau \hat{H}(z)}$, it follows that the characteristic equation is equivalent to
\begin{equation}\label{eq.Q}
Q_\tau^2(z)-\alpha Q_\tau(z)+\beta=0
\end{equation}
The characteristic equation (\ref{eq.char}) has a pair of complex conjugate roots on the imaginary axis if and only if there exists $\omega>0$ such that
$Q(i\omega)$ is a root of the polynomial $P(\lambda)=\lambda^2-\alpha\lambda+\beta$.

\textbf{Case 1: $\alpha^2-4\beta<0$.}

In this case, the polynomial $P(\lambda)$ has complex conjugate roots $Q_\tau(i\omega)$ and $\overline{Q_\tau(i\omega)}$, and hence, Hopf bifurcation may only take place at the equilibrium state $(u^\star, v^\star)$ of system (\ref{sys.wilson.cowan.dd}) along the curve $(\gamma_\tau)$ from the $(\alpha,\beta)$-plane, defined by the following parametric equations:
\begin{equation}\label{sist_gamma}(\gamma_\tau):\quad
\left\{
  \begin{array}{l}
	\ds\alpha =\alpha_\tau(\omega)=2\Re(Q_\tau(i\omega))=\frac{2}{\rho(\omega)}\left[\cos\theta(\omega)-\frac{\omega}{\tau}\sin\theta(\omega)\right]\\
    \ds\beta =\beta_\tau(\omega)=|Q_\tau(i\omega)|^2=\frac{1}{\rho(\omega)^2}\left(1+\frac{\omega^2}{\tau^2}\right)
  \end{array}
\right.\quad ,~\omega>0.
\end{equation}

Indeed, we prove the following: 

\begin{lem}\label{lem.gamma}
If $\alpha^2-4\beta<0$ then the characteristic equation  (\ref{eq.char}) has a pair of complex conjugate roots on the imaginary axis if and only if $(\alpha,\beta)$ belong to the curve $(\gamma_\tau)$ defined by (\ref{sist_gamma}). 
\end{lem}

\begin{proof}
{Since $\alpha^2-4\beta<0$, the} characteristic equation  (\ref{eq.char}) has a pair of complex conjugate roots on the imaginary axis if and only if there exists $\omega>0$ such that the polynomial $P(\lambda)$ has complex conjugate roots $Q_\tau(i\omega)$ and $\overline{Q_\tau(i\omega)}$. As
\begin{align*}
Q_\tau(i\omega)&=\frac{i\omega
+\tau}{\tau\hat{H}(i\omega)}=\frac{i\omega+\tau}{\tau\rho(\omega)e^{-i\theta(\omega)}}=\frac{(i\omega
+\tau)(\cos\theta(\omega)+i\sin\theta(\omega))}{\tau\rho(\omega)}=\\
&=\frac{\tau \cos\theta(\omega)-\omega \sin\theta(\omega)+i(\omega \cos\theta(\omega)+\tau \sin\theta(\omega))}{\tau\rho(\omega)}
\end{align*}
from {Vieta's} formulas, we obtain:
\begin{align*}
\alpha&=Q_\tau(i\omega)+\overline{Q_\tau(i\omega)}=2\Re(Q_\tau(i\omega))=2\cdot\frac{\tau \cos\theta(\omega)-\omega \sin\theta(\omega)}{\tau\rho(\omega)}=\frac{2}{\rho(\omega)}\left[\cos\theta(\omega)-\frac{\omega}{\tau}\sin\theta(\omega)\right];\\
\beta &=Q_\tau(i\omega)\overline{Q_\tau(i\omega)}=|Q_\tau(i\omega)|^2=\left|\frac{i\omega+\tau}{\tau\rho(\omega)e^{-i\theta(\omega)}}\right|^2=\frac{\omega^2+\tau^2}{\tau^2\rho(\omega)^2}=\frac{1}{\rho(\omega)^2}\left(1+\frac{\omega^2}{\tau^2}\right).
\end{align*}
Therefore, we obtain the equivalence from the statement. 
\end{proof}

\textbf{Case 2: $\alpha^2-4\beta\geq 0$.}

In this case, the following preliminary lemma will be needed to establish further results:

\begin{lem}\label{lem.omega}
Let $\tilde{\theta}=\lim\limits_{\omega\rightarrow\infty}\theta(\omega)$ and $K=\lceil\tilde{\theta}/\pi\rceil$ (if $\tilde{\theta}=\infty$ then $K=\infty$). The function
\begin{equation}\label{eq.omega} q_\tau(\omega)=\omega\cos\theta(\omega)+\tau \sin\theta(\omega)
\end{equation}
has exactly $N$ roots $0=\omega_0(\tau)<\omega_1(\tau)<...<\omega_{N-1}(\tau)$, such that $\theta(\omega_n(\tau))\in\left(n\pi-\dfrac{\pi}{2},n\pi\right)$, for {$n\in\{1,2,...,N-1\}$}, where:
$$N=\begin{cases}K, &\text{if }K\in\mathbb{N}~\text{and}~(-1)^K q_\tau(\theta^{-1}(\tilde{\theta}))\leq 0;\\
K+1, &\text{if }K\in\mathbb{N}~\text{and}~(-1)^K q_\tau(\theta^{-1}(\tilde{\theta}))>0;\\
\infty, &\text{if }K=\infty.
\end{cases}$$
\end{lem}

\begin{proof}
From $\theta(0)=0$ it follows that $\omega_0(\tau)=0$ is a root of the function $q_\tau(\omega)$. As $\theta$ is increasing, for any {$n\in\{1,2,...,K\}$} it is easy to see that 
\begin{align*}(-1)^n q_\tau(\omega)<0\quad &\text{for any}~ \omega>0~\text{such that}~ \theta(\omega)\in\left((n-1)\pi,n\pi-\dfrac{\pi}{2}\right)\\
(-1)^n q'_\tau(\omega)>0\quad &\text{for any}~ \omega>0~\text{such that}~\theta(\omega)\in\left(n\pi-\dfrac{\pi}{2},n\pi\right)
\end{align*}
Therefore, for each {$n\in\{1,2,...,K-1\}$}, there exists a unique root $\omega_n(\tau)$ of the function $q_\tau(\omega)$ which satisfies $\theta(\omega_n(\tau))\in\left(n\pi-\dfrac{\pi}{2},n\pi\right)$. The $(K+1)$-th root exists if and only if $(-1)^Kq_\tau(\theta^{-1}(\tilde{\theta}))>0$.
\end{proof}

The following result holds:

\begin{lem}\label{lem.ltau}
If $\alpha^2-4\beta\geq 0$ then the characteristic equation  (\ref{eq.char}) has a pair of complex conjugate roots $\pm i\omega\tau^{-1}$  on the imaginary axis if and only if $\omega$ is a positive root of \eqref{eq.omega}
and $(\alpha,\beta)$ belong to the line \begin{equation}\label{line_tau2}
(l):\quad \beta=\frac{\alpha}{\rho(\omega)\cos\theta(\omega)}-\frac{1}{(\rho(\omega)\cos\theta(\omega))^2},\end{equation}
\end{lem}

\begin{proof}
{Since $\alpha^2-4\beta\geq 0$, the} characteristic equation  (\ref{eq.char}) has a pair of complex conjugate roots on the imaginary axis if and only if there exists $\omega>0$ such that $Q_\tau(i\omega)$ is a real root of the polynomial $P(\lambda)$. Therefore, 
$\Im(Q(i\omega))=0$ and hence, it follows that $\omega$ is a root of the equation $$\omega \cos\theta(\omega)+\tau \sin\theta(\omega)=0.$$
Assuming that such a root exists, let us denote it by $\omega^*$. Hence:
\begin{align*}
Q_\tau(i\omega^*)&=\Re(Q_\tau(i\omega^*))=\frac{\tau \cos\theta(\omega^*)-\omega^* \sin\theta(\omega^*)}{\tau\rho(\omega^*)}=\frac{\tau \cos\theta(\omega^*)+\frac{(\omega^*)^2}{\tau}\cos\theta(\omega^*)}{\tau\rho(\omega^*)}\\
&=\frac{\cos\theta(\omega^*)}{\rho(\omega^*)}\left(1+\frac{(\omega^*)^2}{\tau^2}\right)=\frac{\cos\theta(\omega^*)}{\rho(\omega^*)}\left(1+\frac{\sin^2\theta(\omega^*)}{\cos^2\theta(\omega^*)}\right)=\frac{1}{\rho(\omega^*)\cos\theta(\omega^*)}.
\end{align*}
As $Q_\tau(i\omega^\star)$ is a root of $P(\lambda)$, we obtain:
$$
\beta=\alpha Q_\tau(i\omega^\star)-Q^2_\tau(i\omega^\star)=\frac{\alpha}{\rho(\omega^*)\cos\theta(\omega^*)}-\frac{1}{(\rho(\omega^*)\cos\theta(\omega^*))^2}.
$$
Therefore, the proof is complete. \end{proof}

{ 
\begin{lem}\label{lem.double.root}
The characteristic equation \eqref{eq.char} has a double pure imaginary root if and only if $\alpha^2 = 4 \beta$.
\end{lem}
\begin{proof}
Assume that there exists $\omega>0$ such that $z=i\omega\tau^{-1}$ is a double root of \eqref{eq.char}, i.e.
\begin{equation}\label{sist_delta}
\quad
\left\{
  \begin{array}{l}
	\ds\Delta(z)=(z+1)^2-\alpha H(z)(z+1)+\beta H^2(z)=0\\
    \ds\Delta'(z)=2(z+1)-\alpha H(z) -\alpha (z+1)H'(z)+2\beta H(z)H'(z)=0
  \end{array}
\right.\quad
\end{equation}
Assumption $(A_2)$ guarantees that $H(z)\neq 0$. Eliminating $\beta$ from system \eqref{sist_delta} we obtain the equality:
$$\left(2\frac{z+1}{H(z)} - \alpha \right)\left(1-H'(z)\frac{z+1}{H(z)}\right)  = 0. 
$$
On one hand, if the first term in the above equation is zero, replacing into the first equation of system \eqref{sist_delta}, it follows that $\alpha^2 = 4 \beta$. On the other hand, if the second term of the above equation is zero, as $z = i\omega\tau^{-1}$ and $H(i\omega\tau^{-1})=\rho(\omega)e^{\ds -i\theta(\omega)}$, we obtain:
$$(\omega -i\tau)\left[ \rho'(\omega)-i\rho(\omega)\theta'(\omega)\right] = \rho(\omega)
$$
Taking the imaginary part in the previous equality, leads to  $\tau\rho'(\omega)=-\omega\rho(\omega)\theta'(\omega) $. 
Next, taking the real part, we get
$$ \rho(\omega)=\frac{\omega^2+\tau^2}{\omega}\rho'(\omega).$$
From assumption $(A_2)$, we have $\rho(\omega)>0$ and from $(A_3)$, we have $\rho'(\omega)\leq 0$. Therefore, the above equality is a contradiction. In conclusion, the characteristic equation \eqref{eq.char} has a double pure imaginary root if and only if $\alpha^2 = 4 \beta$. 
\end{proof}

\begin{remark} From Lemma \ref{lem.double.root} it follows that the characteristic equation \eqref{eq.char} has a double pure imaginary root if and only if $(\alpha,\beta)$ are at the intersection points of the curve $(\gamma_\tau)$ with the lines provided by Lemma \ref{lem.ltau} (if any, depending on the existence of the roots of equation \eqref{eq.omega}). These intersection points lie on the parabola $\alpha^2 = 4 \beta$ (as in Proposition \ref{prop.lines.curve}). 
\end{remark}
}

\subsection{Main results}

Combining the previous four lemmas, we deduce that the lines defined for {$n\in\{0,1,...,N-1\}$} by 
\begin{equation}\label{eq.lines}
(l_n):\quad \beta=\mu(\omega_n(\tau))\left[\alpha-\mu(\omega_n(\tau)\right]\quad\text{where}~\mu(\omega_n(\tau))=\frac{1}{\rho(\omega_n(\tau))\cos\theta(\omega_n(\tau))},\end{equation}
as well as the curve $(\gamma_\tau)$ defined parametrically by  \eqref{sist_gamma} play an important role in the bifurcation analysis. Hence, some remarkable properties of these lines and curve are proved in the following: 

\begin{prop}\label{prop.lines.curve}
For the lines $(l_n)$ defined by \eqref{eq.lines} and the curve $(\gamma_\tau)$ defined by \eqref{sist_gamma} the following properties hold: 
\begin{itemize}
    \item[i.] $|\mu(\omega_n(\tau))|$ is increasing with respect to $n$, with $\text{sign}(\mu(\omega_n(\tau)))=(-1)^n$ and {$|\mu(\omega_n(\tau))|>1$};
    \item[ii.] the line $(l_n)$ is included in the region $\alpha^2-4\beta\geq 0$ and is tangent to the parabola $\alpha^2=4\beta$ at the point ${P_n}\left(2\mu(\omega_n(\tau)),\mu(\omega_n(\tau))^2\right)$;  
    \item[iii.] the intersection of the line $(l_n)$ with the vertical axis $\alpha=0$ is the point of $\beta$-coordinate $-\mu(\omega_n(\tau))^2$.
    \item[iv.] the curve $(\gamma_\tau)$ is a simple curve included in the region $\alpha^2-4\beta\leq 0$;
    \item[v.] the line $(l_n)$, the curve $(\gamma_\tau)$ and the parabola $\alpha^2=4\beta$ intersect at the point ${P_n}\left(2\mu(\omega_n(\tau)),\mu(\omega_n(\tau))^2\right)$.
\end{itemize}
\end{prop}

\begin{proof}
For the proof of (i.) From equation \eqref{eq.omega} it easily follows that $$\cos^2(\omega_n(\tau))=\frac{\tau^2}{\omega_n(\tau)^2+\tau^2},$$
and hence: 
$$|\mu(\omega_n(\tau))|=\frac{\sqrt{\omega_n(\tau)^2+\tau^2}}{\tau\rho(\omega_n(\tau))}.$$
Therefore, as $\omega\mapsto\rho(\omega)$ is decreasing, we deduce that the function $n\mapsto |\mu(\omega_n(\tau))|$ is increasing. Moreover, from Lemma \ref{lem.omega} it follows that $\theta(\omega_n(\tau))\in\left(n\pi-\dfrac{\pi}{2},n\pi\right)$ and hence $\text{sign}(\mu(\omega_n(\tau)))=\text{sign}(\cos(\theta(\omega_n(\tau))))=(-1)^n$. 

Moreover, assumption $(A_3)$ implies the function $\beta_\tau(\omega)$ is strictly increasing, hence, the curve $(\gamma_\tau)$ does not cross itself, hence statement (iv.) follows.

The rest of the statements can be easily proved by elementary reasoning.
\end{proof}

We are now ready to prove the main theoretical results, which are based on the previously described root locus technique.

\begin{thm}\label{thm.stab.dom.bounded}
If $N\geq 2$, denoting
$$\omega_\tau=\omega_1(\tau)=\min\{\omega>0~:~\tau \sin\theta(\omega)+\omega \cos\theta(\omega)=0\}\quad\text{and}\quad \mu_\tau=\mu(\omega_1(\tau))=\left(\rho(\omega_\tau)\cos\theta(\omega_\tau)\right)^{-1},$$
the stability region $S_\tau(\alpha,\beta)$ of the equilibrium $(u^\star, v^\star)$ of system \eqref{sys.wilson.cowan.dd} is bounded by the line segments and curve given below:
\begin{align*}
(l_0):&\quad \beta=\alpha-1&, ~\alpha\in\left[1+\mu_\tau,2\right];\\
(l_\tau):&\quad \beta=\mu_\tau(\alpha-\mu_\tau) &,~\alpha\in\left[2\mu_\tau,1+\mu_\tau\right];\\
(\gamma_\tau):&\quad \left\{
  \begin{array}{l}
    \ds\alpha =\displaystyle\frac{2}{\rho(\omega)}\left[\cos\theta(\omega)-\frac{\omega}{\tau}\sin\theta(\omega)\right]\\
    \ds\beta =\displaystyle\frac{1}{\rho(\omega)^2}\left(1+\frac{\omega^2}{\tau^2}\right)
  \end{array}
\right.&,~\omega\in (0,\omega_\tau).
\end{align*}
At the boundary of the stability region $S_\tau(\alpha,\beta)$, the following bifurcation phenomena take place in a neighborhood of the equilibrium $(u^\star, v^\star)$ of system \eqref{sys.wilson.cowan.dd}:
\begin{itemize}
\item[a.] Saddle-node bifurcations take place along the open line segment $(l_0)$;
\item[b.] Hopf bifurcations take place along the open line segment $(l_\tau)$ and curve $(\gamma_\tau)$;
\item[c.] Bogdanov-Takens bifurcation at $(\alpha,\beta)=(2,1)$;
\item[d.] Double-Hopf bifurcation at $(\alpha,\beta)=\left(2\mu_\tau,\mu_\tau^2\right)$;
\item[e.] Zero-Hopf bifurcation $(\alpha,\beta)=\left(1+\mu_\tau,\mu_\tau\right)$.
\end{itemize}
\end{thm}

\begin{proof}

It is easy to see that for $\alpha=\beta=0$, the characteristic equation has negative real roots, and therefore, the equilibrium state $(u^\star, v^\star)$ of system \eqref{sys.wilson.cowan.dd} is asymptotically stable. As the roots of the characteristic function $\Delta(z)$ (or $\hat{\Delta}(z)$) continuously depend on the parameters $\alpha$ and $\beta$, the number of roots $z$ such that $\Re(z)>0$ may change if and only if a root $z=0$ or a pair of pure imaginary roots appear, i.e. along the curve $(\gamma_\tau)$ or the lines  $(l_n)$, {$n\in\{0,1,...,N-1\}$} in the $(\alpha,\beta)$-plane. A simple geometrical reasoning based on Proposition \ref{prop.lines.curve} shows that the line segments and curve segment given in the statement above enclose a connected region of the $(\alpha,\beta)$-plane containing the origin, i.e. the stability region $S_\tau(\alpha,\beta)$ of the equilibrium $(u^\star, v^\star)$ of system \eqref{sys.wilson.cowan.dd}. Moreover, Theorem \ref{thm.saddle.node} shows that a saddle-node bifurcation takes place along the line segment $(l_0)$, and hence, statement a. holds.

b. We will first show that Hopf bifurcations take place along the curve segment $(\gamma_\tau)$, with $\omega\in(0,\omega_\tau)$. 

Lemma \ref{lem.gamma} provides that the characteristic equation (\ref{eq.Q}) has a pair of pure imaginary roots $\pm i\omega$ if  $(\alpha,\beta)$ belong to the curve $(\gamma_\tau)$. Let us consider an arbitrary $\omega\in(0,\omega_\tau)$ and denote by $z(\alpha,\beta)$ the root of the characteristic equation (\ref{eq.Q}) satisfying $z(\alpha^*,\beta^*)=i\omega$ where $(\alpha^*,\beta^*)=(\alpha_\tau(\omega),\beta_\tau(\omega))\in (\gamma_\tau)$. Our aim is to prove the transversality condition:
$$\nabla_{\overline{u}} \Re(z)(\alpha^*,\beta^*)>0,$$
for any outward pointing vector $\overline{u}$ from the region $S_\tau(\alpha,\beta)$, i.e. $\langle \overline{u} ,\overline{n}\rangle>0$, where $\overline{n}$ denotes the outward pointing normal vector to the curve $(\gamma_\tau)$. 

From (\ref{eq.Q}) it follows that $z(\alpha,\beta)$ is a solution of the equation
\begin{equation}\label{eq.dem.Q}
Q_\tau^2(z)-\alpha Q_\tau(z)+\beta=0.
\end{equation}
Taking the derivative with respect to $\alpha$, we obtain:
$$2Q_\tau(z)Q_\tau'(z)\frac{\partial z}{\partial\alpha}-Q_\tau(z)-\alpha Q_\tau'(z)\frac{\partial z}{\partial\alpha}=0,$$
and we express:
\begin{equation}\label{eq.dz.da}
\frac{\partial z}{\partial\alpha}=\frac{Q_\tau(z)}{Q_\tau'(z)[2Q_\tau(z)-\alpha]}.\end{equation}
Therefore, based on the parametric equations of the curve $(\gamma_\tau)$ given by (\ref{sist_gamma}), we deduce:
$$
\frac{\partial z}{\partial\alpha}\bigg\rvert_{(\alpha^*,\beta^*)} =\frac{Q_\tau(i\omega)}{Q_\tau'(i\omega)[2Q_\tau(i\omega)-\alpha^*]}=\frac{Q_\tau(i\omega)}{Q_\tau'(i\omega)[2Q_\tau(i\omega)-2\Re(Q_\tau(i\omega))]}
=\frac{Q_\tau(i\omega)}{2Q_\tau'(i\omega)\cdot i\cdot\Im(Q_\tau(i\omega))}.$$
Applying the real part, we finally obtain:
$$\frac{\partial \Re(z)}{\partial\alpha}\bigg\rvert_{(\alpha^*,\beta^*)}=\Re\left[-\frac{i}{2}\cdot\frac{Q_\tau(i\omega)}{Q_\tau'(i\omega)\Im(Q_\tau(i\omega))}\right]=\frac{1}{2}\cdot \Im\left[\frac{Q_\tau(i\omega)}{Q_\tau'(i\omega)\Im(Q_\tau(i\omega))}\right]=\frac{1}{2\Im( Q_\tau(i\omega))}\cdot\Im\left[\frac{Q_\tau(i\omega)}{Q_\tau'(i\omega)}\right].
$$
Taking the derivative with respect to $\beta$ in equation (\ref{eq.dem.Q}), we obtain:
$$2Q_\tau(z)Q_\tau'(z)\frac{\partial z}{\partial\beta}-\alpha Q_\tau'(z)\frac{\partial z}{\partial\beta}+1=0,$$
and we express:
\begin{equation}\label{eq.dz.db}\frac{\partial z}{\partial\beta}=\frac{1}{Q_\tau'(z)[\alpha-2Q_\tau(z)]}.
\end{equation}
Again, using the expressions of the parametric curve $(\gamma_\tau)$ given by (\ref{sist_gamma}), we deduce:
$$\frac{\partial z}{\partial\beta}\bigg\rvert_{(\alpha^*,\beta^*)}=\frac{1}{2Q_\tau'(i\omega)[-i\Im( Q_\tau(i\omega))]}=\frac{i}{2Q_\tau'(i\omega)\Im( Q_\tau(i\omega))}.$$
Once again, applying the real part, we arrive at:
$$\frac{\partial\Re(z)}{\partial\beta}\bigg\rvert_{(\alpha^*,\beta^*)}=-\frac{1}{2}\cdot\Im\left[\frac{1}{Q_\tau'(i\omega)\Im(Q_\tau(i\omega))}\right]=-\frac{1}{2\Im(Q_\tau(i\omega))}\cdot\Im\left[\frac{1}{Q_\tau'(i\omega)}\right]$$
and thus, we obtain the gradient vector:
\begin{align*}\nabla \Re(z)(\alpha^*,\beta^*)&=\left(\frac{\partial\Re(z)}{\partial\alpha},\frac{\partial\Re(z)}{\partial\beta}\right)\bigg\rvert_{(\alpha^*,\beta^*)}=\frac{1}{2\Im(Q_\tau(i\omega))}
\left(\Im\Big[\frac{Q_\tau(i\omega)}{Q_\tau'(i\omega)}\Big],-\Im\left[\frac{1}{Q_\tau'(i\omega)}\right]\right)=\\
&=\frac{1}{2\Im(Q_\tau(i\omega))|Q_\tau'(i\omega)|^2}
\left(\Im\left[Q_\tau(i\omega)\overline{Q_\tau'(i\omega)}\right],\Im( Q_\tau'(i\omega))\right).
\end{align*}
The tangent vector to the curve $(\gamma_\tau)$ at the point $(\alpha^*,\beta^*)$ is $(\alpha'(\omega),\beta'(\omega))$, where
\begin{align*}
\alpha'(\omega)&=\frac{d}{d\omega}(2\Re(Q_\tau(i\omega)))=2\Re\left[\frac{d}{d\omega}Q_\tau(i\omega)\right]=2\Re(i\cdot Q_\tau'(i\omega))=-2\Im(Q_\tau'(i\omega))\\
\beta'(\omega)&=\frac{d}{d\omega}|Q_\tau(i\omega)|^2=\frac{d}{d\omega
}\Big[Q_\tau(i\omega)\overline{Q_\tau(i\omega)}\Big]=i\cdot Q_\tau'(i\omega)\overline{Q_\tau(i\omega)}+Q_\tau(i\omega)\overline{i\cdot Q_\tau'(i\omega)}=\\
&=i\cdot Q_\tau'(i\omega)\overline{Q_\tau(i\omega)}-i\cdot Q_\tau(i\omega)\overline{Q_\tau'(i\omega)}=i\cdot( Q_\tau'(i\omega)\overline{Q_\tau(i\omega)}-Q_\tau(i\omega)\overline{Q_\tau'(i\omega)})=\\
&=-2\Im\Big[Q_\tau'(i\omega)\overline{Q_\tau(i\omega)}\Big]=2\Im\left[Q_\tau(i\omega)\overline{Q_\tau'(i\omega)}\right].
\end{align*}
Thus, we obtain the following tangent vector to the curve $(\gamma_\tau)$: 
$$(\alpha'(\omega),\beta'(\omega))=2\left(-\Im(Q_\tau'(i\omega)),\Im\left[Q_\tau(i\omega)\overline{Q_\tau'(i\omega)}\right]\right).$$
Therefore, fixing the orientation of the curve $(\gamma_\tau)$ in the direction of increasing $\omega$, a right-pointing normal vector to the curve $(\gamma_\tau)$ is:
$$\overline{n}(\omega)=\left(\Im\Big[Q_\tau(i\omega)\overline{Q_\tau'(i\omega)}\Big],\Im(Q_\tau'(i\omega))\right).$$
Hence, the directional derivative is:
$$\nabla_{\overline{u}} \Re(z)(\alpha^*,\beta^*)=\left\langle\nabla \Re(z)(\alpha^*,\beta^*),\overline{u}\right\rangle=\langle\frac{\overline{n}(\omega)}{2\Im (Q_\tau(i\omega))|Q'_\tau(i\omega)|^2},\overline{u}\rangle
=\frac{{\tau\rho(\omega)}}{2q_\tau(\omega)|Q'_\tau(i\omega)|^2}\langle\overline{n}(\omega),\overline{u}\rangle>0,
$$
as { $\rho(\omega)>0$}, $\langle\overline{n}(\omega),\overline{u}\rangle>0$ and $q_\tau(\omega)>0$ for any $\omega\in(0,\omega_\tau)$, from the proof of Lemma \ref{lem.omega}. Therefore, the transversality condition holds, and it follows that a Hopf bifurcation takes place along the curve segment $(\gamma_\tau)$ which bounds the stability region $S_\tau(\alpha,\beta)$. 

{  In general, along other segments of the curve $(\gamma_\tau)$, it is important to emphasize that if $\omega\in(\omega_{n}(\tau),\omega_{n+1}(\tau))$, where $n\in\{0,1,...,N-1\}$ and  $z(\alpha,\beta)$ denotes the root of the characteristic equation (\ref{eq.Q}) satisfying $z(\alpha^*,\beta^*)=i\omega$ where $(\alpha^*,\beta^*)=(\alpha_\tau(\omega),\beta_\tau(\omega))\in (\gamma_\tau)$, by a similar reasoning as above, we get
$$\nabla_{\overline{u}} \Re(z)(\alpha^*,\beta^*)
=\frac{{\tau\rho(\omega)}}{2q_\tau(\omega)|Q'_\tau(i\omega)|^2}\langle\overline{n}(\omega),\overline{u}\rangle,
$$
where the orientation of the curve $(\gamma_\tau)$ is fixed in the direction of increasing $\omega$, and $\overline{n}(\omega)$ is a right-pointing normal vector to the curve $(\gamma_\tau)$ (e.g. see Figure \ref{fig.bif.curves}). From Lemma \ref{lem.omega}, it follows that $\text{sign}(q_\tau(\omega))=(-1)^n$ and we can also notice that $\text{sign}(\langle\overline{n}(\omega),\overline{u}\rangle)=(-1)^n$, for any vector $\overline{u}$ pointing towards the region above the curve segment of $(\gamma_\tau)$, with $\omega\in(\omega_{n}(\tau),\omega_{n+1}(\tau))$. Therefore, $\nabla_{\overline{u}} \Re(z)(\alpha^*,\beta^*)>0$, and hence,  stability cannot be regained when crossing the curve segment from the connected region below the curve segment to the region above.
}
\begin{figure}[ht]
    \centering
	\includegraphics[width=0.5\linewidth]{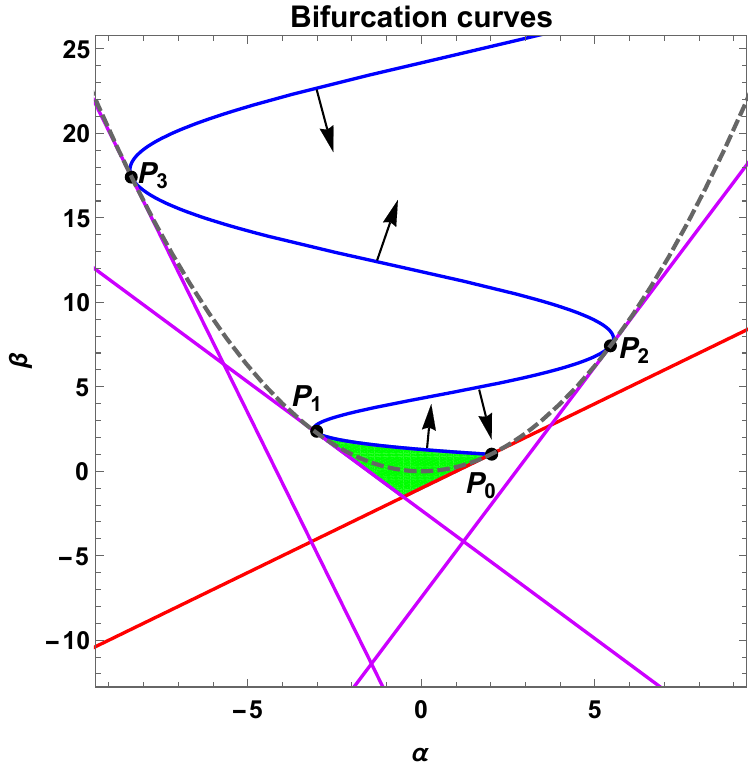}
	\vspace*{-0.5cm}
	\caption{{Bifurcation curves exemplified for the Dirac kernel case with $\tau = 2$: $(\gamma_\tau)$ is the blue curve, $(l_0)$ is the red line, $(l_1)$, $(l_2)$ and  $(l_3)$ are the purple lines,  the parabola $\alpha^2=4\beta$ is the gray dotted curve. The stability region is plotted in green. The black arrows represent right-pointing normal vectors to the $(\gamma_\tau)$-curve, where the orientation of the curve is fixed in the direction of increasing $\omega$. The intersection points $P_n$ are given by Proposition \ref{prop.lines.curve}.}}
	\label{fig.bif.curves}
\end{figure}

In the following, we will show that Hopf bifurcations take place along the line segment $(l_\tau)$, with $\alpha\in\left[2\mu_\tau,1+\mu_\tau\right]$. If $(\alpha,\beta)$ belong to this line segment, Lemma \ref{lem.ltau} shows that the characteristic equation (\ref{eq.Q}) has a pair of pure imaginary roots $\pm i\omega_\tau$. Let us denote by $z(\alpha,\beta)$ the root of the characteristic equation (\ref{eq.Q}) with the property $z(\alpha^*,\beta^*)=i\omega_\tau$ where $\ds\alpha^*\in\left[2\mu_\tau,1+\mu_\tau\right]$ is arbitrarily fixed and $\ds\beta^*=\mu_\tau(\alpha-\mu_\tau)$. We will prove the transversality condition
$$\nabla_{\overline{u}} \Re(z)(\alpha^*,\beta^*)>0,$$
for any vector $\overline{u}$ pointing outward from the region $S_\tau(\alpha,\beta)$, i.e. $\langle \overline{u} ,\overline{n}\rangle>0$, where $\ds\overline{n}=\left(\mu_\tau,-1\right)$ is an outward pointing normal vector to the line segment $(l_\tau)$.

Similarly as in (\ref{eq.dz.da}) and (\ref{eq.dz.db}) we have:
$$\frac{\partial z}{\partial\alpha}=\frac{Q_\tau(z)}{Q_\tau'(z)[2Q_\tau(z)-\alpha]}\qquad\text{and}\qquad \frac{\partial z}{\partial\beta}=\frac{1}{Q_\tau'(z)[\alpha-2Q_\tau(z)]}.$$
Using the fact that, $\ds Q_\tau(z)=\frac{z+\tau}{\tau \hat{H}(z)}$, we deduce:
$$
Q'_\tau(z)=\frac{\hat{H}(z)-(z+\tau)\hat{H}'(z)}{\tau\hat{H}^2(z)}=\frac{1}{\tau\hat{H}(z)}-\frac{z+\tau}{\tau\hat{H}(z)}\cdot\frac{\hat{H}'(z)}{\hat{H}(z)}=\frac{1-\tau Q_\tau(z)\hat{H}'(z)}{\tau\hat{H}(z)}
$$
and therefore, using the fact that $Q_\tau(i\omega_\tau)=\mu_\tau$ (as in the proof of Lemma \ref{lem.ltau}), we obtain
$$
Q'_\tau(i\omega_\tau)=\frac{1-\tau\mu_\tau e^{-i\theta(\omega_\tau)}\left[-i\rho'(\omega_\tau)-\rho(\omega_\tau)\theta'(\omega_\tau)\right]}{\tau\rho(\omega_\tau)e^{-i\theta(\omega_\tau)}}=\frac{e^{i\theta(\omega_\tau)}+\tau\mu_\tau[i\rho'(\omega_\tau)+\rho(\omega_\tau)\theta'(\omega_\tau)]}{\tau\rho(\omega_\tau)}.
$$
Hence, it is easy to see that:
$$\Re \left[Q'_\tau(i\omega_\tau)\right]=\frac{\cos\theta(\omega_\tau)+\tau\mu_\tau\rho(\omega_\tau)\theta'(\omega_\tau)}{\tau\rho(\omega_\tau)}=\mu_\tau(\tau^{-1}\cos^2\theta(\omega_\tau)+\theta'(\omega_\tau))<0,
$$
as $\mu_\tau<0$. The following gradient vector is obtained:
\begin{align*}
\nabla\Re(z)(\alpha^*,\beta^*)&=\Re\left[\left(\frac{Q_\tau(i\omega_\tau)}{Q_\tau'(i\omega_\tau)[2Q_\tau(i\omega_\tau)-\alpha^*]},\frac{1}{Q_\tau'(i\omega_\tau)[\alpha^*-2Q_\tau(i\omega_\tau)]}\right)\right]=\\
&=\frac{1}{\left(2\mu_\tau-\alpha^*\right)}\cdot\Re\left[\left(\frac{\mu_\tau}{Q'_\tau(i\omega_\tau)},-\frac{1}{Q'_\tau(i\omega_\tau)}\right)\right]=
\\&=\frac{1}{\left(2\mu_\tau-\alpha^*\right)}\cdot\Re\left(\frac{1}{Q'_\tau(i\omega_\tau)}\right)\cdot(\mu_\tau,-1)=
\\&=\frac{\Re\left(Q'_\tau(i\omega_\tau)\right)}{\left(2\mu_\tau-\alpha^*\right)|Q'_\tau(i\omega)|^2}\cdot\overline{n}.
\end{align*}
As the scalar term appearing in front of $\overline{n}$ is positive, it follows that the gradient vector is indeed a normal vector to the line $(l_\tau)$ pointing outward from the stability region $S_\tau(\alpha,\beta)$. In conclusion, the transversality condition is now deduced as in the case of the curve segment $(\gamma_\tau)$, and it follows that a Hopf bifurcation takes place along the line segment $(l_\tau)$. { We also emphasize that the transversality condition may be checked in a similar way as before along each line $(l_n)$, $n\in\{1,2,...,N-1\}$.}

The points c., d. and e. from the statement of the theorem follow easily taking into consideration the intersections between the lines $(l_0)$, $(l_\tau)$ and the curve $(\gamma_\tau)$. 
\end{proof}

\begin{exmp}\label{ex.dirac}
If a \textbf{Dirac kernel} $h(t)=\delta(t-\tau)$ is considered in \eqref{sys.wilson.cowan.dd}, we have $\rho(\omega)=1$ and $\theta(\omega)=\omega$ and it can be shown that $S_{\tau_2}(\alpha,\beta)\subset S_{\tau_1}(\alpha,\beta)$, for any $\tau_1<\tau_2$. Hence, the intersection of all stability regions in this case will be: \[S_{\infty}(\alpha,\beta)=\left\{(\alpha,\beta)\in\mathbb{R}^2:|\alpha|-1<\beta< 1\right\},\]
which is shown in the first panel of Fig. \ref{fig.stab.dom} as the shaded triangle.
\end{exmp}

\begin{exmp}\label{ex.strong.gamma}
If a \textbf{strong Gamma kernel} $h(t)=4\tau^{-2}t\exp(-2t/\tau)$ is considered in system \eqref{sys.wilson.cowan.dd}, as $\rho(\omega)=4/(4+\omega^2)$ and $\theta(\omega)=2\arctan(\omega/2)$, 
the explicit equation of the curve $(\gamma_\tau)$ is: \begin{equation}\label{curve.gamma.strong}
(\gamma_\tau):\quad \beta=\frac{\left[4(\tau+2)-\alpha\tau\right]^2 \left[(\tau+2)^2-2 \alpha\right]}{4 \tau (\tau+4)^3}\quad,~\text{for}~
\alpha<2.   
\end{equation}
Moreover, as $\mu_\tau=-(\tau+2)^2/\tau$, the equation of the line $(l_\tau)$ is
\begin{equation}\label{line.gamma.strong}
(l_\tau):\quad \beta=-\frac{(\tau+2)^2}{\tau} \left(\alpha+\frac{(\tau+2)^2}{\tau}\right).  
\end{equation}
Therefore, the stability region given by Theorem \ref{thm.stab.dom.bounded} is
\[S_\tau(\alpha,\beta)=\left\{(\alpha,\beta)\in\mathbb{R}^2:\beta>-\frac{(\tau+2)^2}{\tau} \left(\alpha+\frac{(\tau+2)^2}{\tau}\right)~\text{and }~\alpha-1<\beta<\frac{\left[4(\tau+2)-\alpha\tau\right]^2 \left[(\tau+2)^2-2 \alpha\right]}{4 \tau (\tau+4)^3}\right\}.\]
Using basic inequality techniques, it can be shown that $S_2(\alpha,\beta)\subset S_\tau(\alpha,\beta)$, for any $\tau>0$, where: 
\[S_2(\alpha,\beta)=\left\{(\alpha,\beta)\in\mathbb{R}^2:\beta>-8 \left(\alpha+8\right)~\text{and }~\alpha-1<\beta<\left(\frac{8-\alpha}{6}\right)^3\right\}.\]
In conclusion, in the case of a strong Gamma kernel, $S_2(\alpha,\beta)$ (see third panel of Fig. \ref{fig.stab.dom}) represents the intersection of all stability regions $S_\tau(\alpha,\beta)$, for any $\tau>0$. 
\end{exmp}

The following result follows similarly in the case when equation (\ref{eq.omega}) does not admit any positive roots. In this case, the stability region $S_\tau(\alpha,\beta)$ will be unbounded. 

\begin{thm}\label{thm.stab.dom.unbounded}
If equation (\ref{eq.omega}) does not admit any strictly positive real root (i.e. $N=1$), the boundary of the stability region $S_\tau(\alpha,\beta)$ of the equilibrium state $(u^\star, v^\star)$ of system \eqref{sys.wilson.cowan.dd} is given by the union of the half-line and curve given below:
\begin{align*}
(l_0):&\quad \beta=\alpha-1&, ~\alpha\in (-\infty,2];\\
(\gamma_\tau):&\quad \left\{
  \begin{array}{l}
    \ds\alpha =\displaystyle\frac{2}{\rho(\omega)}\left[\cos\theta(\omega)-\frac{\omega}{\tau}\sin\theta(\omega)\right]\\
    \ds\beta =\displaystyle\frac{1}{\rho^2(\omega)}\left(1+\frac{\omega^2}{\tau^2}\right)
  \end{array}
\right.&,~\omega>0.
\end{align*}
At the boundary of the stability region $S_\tau(\alpha,\beta)$, the following bifurcation phenomena take place in a neighborhood of the equilibrium $(u^\star, v^\star)$ of system \eqref{sys.wilson.cowan.dd}:
\begin{itemize}
\item[a.] Saddle-node bifurcations take place along the open half-line $(l_0)$;
\item[b.] Hopf bifurcations take place along the curve $(\gamma_\tau)$;
\item[c.] Bogdanov-Takens bifurcation at $(\alpha,\beta)=(2,1)$.
\end{itemize}
\end{thm}

\begin{exmp}\label{ex.weak.gamma}
If a \textbf{weak Gamma kernel} $h(t)=\tau^{-1}\exp(-t/\tau)$ is considered in system \eqref{sys.wilson.cowan.dd}, as $\rho(\omega)=(1+\omega^2)^{-1/2}$ and $\theta(\omega)=\arctan\omega$, it can be easily seen that equation  \eqref{eq.omega} does not have any positive roots. The explicit equation of the curve $(\gamma_\tau)$ (which is in fact an arc of a parabola) is: \begin{equation}\label{curve.gamma.weak}
(\gamma_\tau):\quad \beta=\left(1-\frac{\alpha}{2}\right)^2+\left(\tau+\frac{1}{\tau}\right)\left(1-\frac{\alpha}{2}\right)+1,\quad\text{for }\alpha<2.   
\end{equation}
Therefore, the stability region given by Theorem \ref{thm.stab.dom.unbounded} is
\[S_\tau(\alpha,\beta)=\left\{(\alpha,\beta)\in\mathbb{R}^2:\alpha<2\text{ and }\alpha-1<\beta<\left(1-\frac{\alpha}{2}\right)^2+\left(\tau+\frac{1}{\tau}\right)\left(1-\frac{\alpha}{2}\right)+1 \right\}.\]
As $\tau+\dfrac{1}{\tau}\geq 2$, it follows that $S_1(\alpha,\beta)\subset S_\tau(\alpha,\beta)$, for any $\tau>0$. Hence, in this case, $S_1(\alpha,\beta)$ (shown in the second panel of Fig. \ref{fig.stab.dom}) represents the intersection of all stability regions $S_\tau(\alpha,\beta)$, for all $\tau>0$. 
\end{exmp}

\begin{figure}[http]
	\centering
	\includegraphics[width=\linewidth]{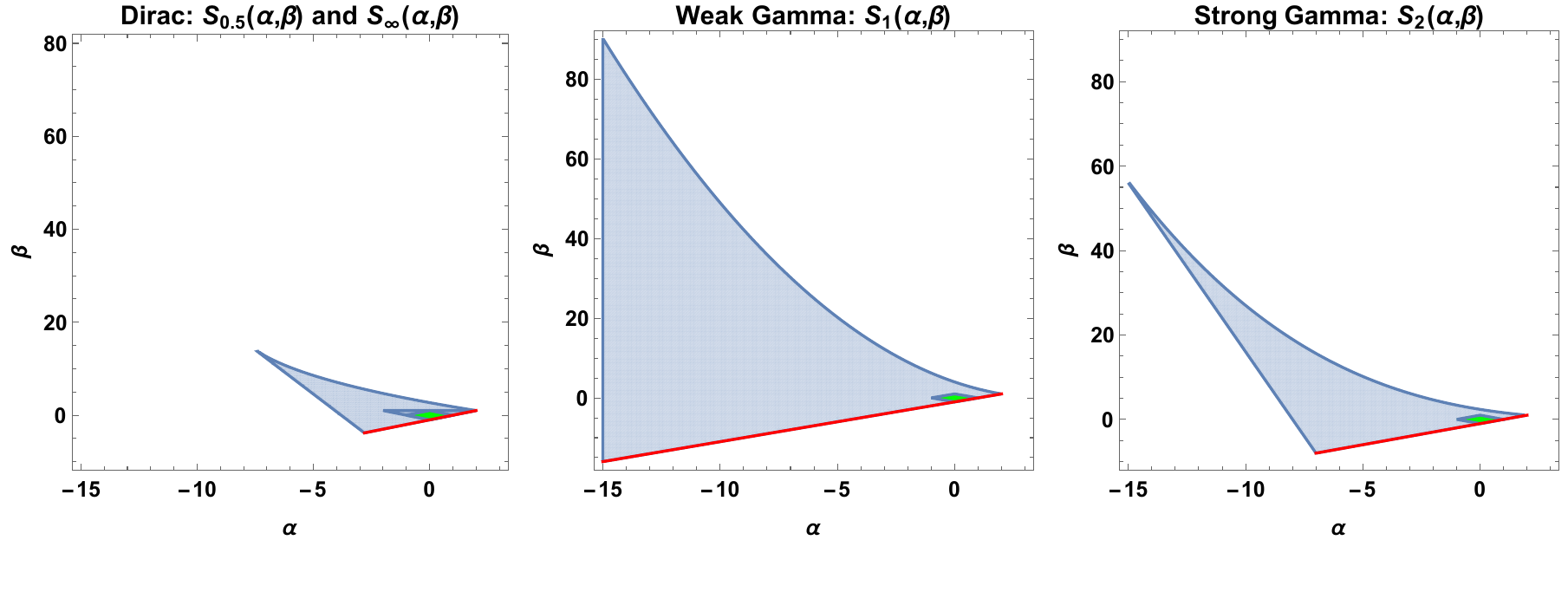}
	\caption{Stability regions
	for different types of delay kernels obtained by Theorems  \ref{thm.stab.dom.bounded} and \ref{thm.stab.dom.unbounded}. The green region represents a delay-kernel-invariant subset of $S_\tau(\alpha,\beta)$ provided by Theorem \ref{thm.stab}. Hopf bifurcations occur along the blue curves and line segments, while the red line corresponds to saddle-node bifurcations.}
	\label{fig.stab.dom}
\end{figure}

{ Comparable stability characterisations have been obtained for networks of identical neurons in \cite{Jessop_2010}, without specifically focusing on the occurrence of bifurcation phenomena.  However, in this previous work, the stability region is characterized in terms of the eigenvalues of the network's connection matrix. While our results from Theorems \ref{thm.stab.dom.bounded} and \ref{thm.stab.dom.unbounded} can be easily translated in similar terms, for our two-dimensional Wilson-Cowan model it seems to be more convenient to use $\alpha$ and $\beta$ as characteristic parameters, allowing us to directly discuss the combined effect of the coupling coefficients and the average time delay on the stability regime of the considered model, as described below.}

While Theorem~\ref{thm.stab} guarantees that there is a {\it kernel-invariant} stability region common to a large class of delay distributions (including all Dirac and Gamma kernels), Figure~\ref{fig.stab.dom} shows that the size of the whole stability region varies widely around this region, depending on the kernel type. 
As an illustration of Examples \ref{ex.dirac}, \ref{ex.strong.gamma} and \ref{ex.weak.gamma}, the panels of Figure \ref{fig.stab.dom} show respectively the intersection of stability regions (over all $\tau>0$) found theoretically for the  Dirac kernel (dark triangle $S_{\infty}(\alpha,\beta)$ in left panel), the weak Gamma kernel (shaded region $S_1(\alpha,\beta)$ in the center panel) and for the strong Gamma kernel (shaded region $S_2(\alpha,\beta)$ in the right panel). It is then easy to see that the Gamma kernels support a larger subset of parameters for which the system converges to a stable equilibrium than the Dirac kernel. The weak Gamma kernel in particular produces an unbounded stability region (according to Theorem \ref{thm.stab.dom.unbounded}). Since the subsequent analysis in the original Wilson-Cowan model  was done for a weak Gamma kernel~\cite{Wilson-Cowan}, one interesting avenue for future work would be to compare the behaviors between the model with Gamma kernel as it was adapted in the original paper (with coarse-grain approximation), and in its unaltered form (without the coarse-grain approximation).

Our results also reveal that, for each kernel type, there are multiple mechanisms (relying on changes to both the parameters $(\alpha,\beta)$ and to the delay $\tau$) that may be used by the system to move from a regime with a unique stable equilibrium (steady firing activity) to a regime of stable oscillations. A transition based solely on altering the integration parameters $(\alpha,\beta)$ is more easily accomplished for the Dirac kernel than for the strong Gamma kernel, since the latter has a significantly larger stability region. For example, smaller perturbation needed to be applied to $(\alpha,\beta) = (0,0)$ (which is in the kernel-invariant stability region) in order to produce oscillations for the Dirac kernel versus the strong Gamma kernel. Moreover, notice that in both the case of the Dirac and the strong Gamma kernel, the transition in and out of the stability region may be promoted by either an increase in $\beta$, or by a decrease in $\alpha$. In contrast, in the case of the weak Gamma kernel (for which the stability region is unbounded), changing $\beta$ does not alter the stability of the equilibrium (i.e., does not transition the system into { and oscillatory} regime) if $\alpha$ is negative and sufficiently large in absolute value.

Physiologically, these paths can be tied back to changes in the coupling weights $a,b,c,d$, since $\alpha=af_1'(\theta_u+au^\star+bv^\star)+df_2'(\theta_v+cu^\star+dv^\star)$ and $\beta =(ad-bc)f_1'(\theta_u+au^\star+bv^\star)f_2'(\theta_v+cu^\star+dv^\star)$. While the dependence of $(\alpha,\beta)$ on $(a,b,c,d)$ is too complex to draw any direct interpretations, increasing the cross-population coupling coefficients $b$ and $c$ may be  way of lowering $\beta$ values, while self-coupling coefficients $a$ and $d$ seems to promote increases in both $\alpha$ and $\beta$. It can then be speculated that, in the case of Dirac and strong Gamma kernels, transitions in and out of { oscillations} can be accomplished by changing either self-coupling or inter-coupling strength (to cross either one of the Hopf curves delimiting { stable state from oscillatory} regimes). More explicit relationships can be worked out between the computational parameters $(\alpha,\beta)$ and the connectivity weights $(a,b,c,d)$ under specific modeling assumptions for the integrating functions $f_1$ and $f_2$, the type of kernel and distributed delay $\tau$. To illustrate this possibility, we explored numerically a particular theoretical example (in Section~\ref{application1}), and a neurophysiology-driven application to dynamic behavior in the basal ganglia (in Section~\ref{application2}).

More generally, our theoretical work suggests that different types of delay distributions may be optimal for different dynamic behaviors and transitions between them, and supports the relatively unexplored possibility that the delay kernel may act as a potential specialization mechanism of the cell to perform circuit-specific and task-specific functions. For example, functions which generally require constant firing rates may be easily performed by neural populations specialized in weak Gamma integration of inputs. Functions that require increased flexibility in and out of { oscillations} may be best accomplished by strong Gamma integrators. Functions characterized by transitions between different and more complex { oscillatory patterns} may require a Dirac cell specialization (as additionally illustrated by our simulations in the next section). This will be further discussed and contextualized in the Discussion section.

\color{black}

\section{Applications and numerical simulations}

In our previous numerical work on the model~\cite{kaslik2020wilson}, we showed how periodic behaviors, but also quasi-periodic and chaotic patterns are also possible in the $(u,v)$-phase-plane. In this study, we aim to further illustrate how access to these behaviors, as well as the transitions between them, are governed not only by the coupling strengths and by the delay $\tau$, but also by the shape of the delay distribution. { To this end, we performed numerical simulations, using our in house codes written in Wolfram Mathematica~\cite{Mathematica}, and made available on github~\cite{github}.} 

\subsection{Simulations in a theoretical context}
\label{application1}

In this section, a brief numerical experiment aims to simulate the kinetic details in a series of behaviors that were classified directly by our theoretical results. To facilitate comparison with the original reference, we considered the same activation functions $f_1(x) = f_2(x) = f(x)=(1+\exp(-\delta x))^{-1}$ and input drives $\theta_u=0.1$, $\theta_v=0.2$. The coupling parameters in our current simulation are fixed to new values $a=d=-19$, $b=c=10$, and $\delta=10$, corresponding overall to new characteristic parameters $\alpha=-17.8796$ and $\beta=57.7268$. 

{ As $\alpha^2>4\beta$ and $\beta>\alpha-1$ from the theoretical analysis presented in the previous section, it follows that a Hopf bifurcation occurs only in the conditions of Theorem \ref{thm.stab.dom.bounded}, when $(\alpha,\beta)\in (l_\tau)$. To find the critical value of the time delay $\tau$ for a particular delay kernel, we proceed as follows, according to Theorem \ref{thm.stab.dom.bounded}: 
\begin{itemize}
    \item find the roots $\mu$ of the quadratic equation $\beta=\mu(\alpha-\mu)$ in the interval $(-\infty,-1)$;
    \item for each $\mu$ computed at the previous step, numerically compute the root $\omega>0$ of the equation $\rho(\omega)\cos\theta(\omega)=\mu^{-1}$ such that $\theta(\omega)\in(\pi/2,\pi)$, and retain the smallest such value $\omega^*$; 
    \item the critical value of the average time delay is $\tau^*=-\omega^*\cot\theta(\omega^*)$.
\end{itemize}
It is also important to note that the frequency of oscillations at the Hopf bifurcation which occurs at the critical value $\tau^*$ of the average time delay is computed according to the formula $\dfrac{\omega^*}{2\pi\tau^*}$, based on the fact that, according to our theoretical results, $\pm i\omega^*/\tau^*$ are the critical roots of the characteristic function $\Delta(z)$.
}

While for these fixed parameter values the system always has a steady state at $(u^\star,v^\star)=(0.0478985, 0.0511112)$, the stability of this equilibrium changes when the time delay $\tau$ is increased past a supercritical Hopf bifurcation value $\tau = \tau^*$, with the creation of a stable limit cycle (that reflects onset of { oscillations} in the corresponding cells). The position of $\tau^*$, as well as the geometry, duty cycle and evolution of the oscillation born through this bifurcation further depend on the type of kernel used (as shown in Figure~\ref{fig.simulations}). For example, based on Theorem \ref{thm.stab.dom.bounded}, the critical value is $\tau^\star_0=0.120766$ for the Dirac kernel, $\tau^\star_2=0.433992$ for the strong Gamma kernel.  { The frequencies of oscillations occurring due to the Hopf bifurcation at the critical values of the average time delay are $2.16675$ (for the Dirac kernel) and $0.87829$ (for the strong Gamma kernel), respectively (see Fig. \ref{fig.simulations}).
As for the values of $\alpha$ and $\beta$ given above, we have $\alpha^2>4\beta$ and $\beta>\alpha-1$, Theorem \ref{thm.stab.dom.unbounded} guarantees that, for the specific case of a weak Gamma kernel}, a Hopf bifurcation never occurs for any $\tau>0$ (the equilibrium  $(u^\star,v^\star)$  stays locally asymptotically stable). Hence, no oscillations are expected to occur in a neighborhood of the equilibrium if a weak Gamma kernel is considered in the mathematical model. 

This reflects the potential importance of using different delay kernel types when modeling coupled population activity, since different delay distribution profiles may promote different behaviors. For example, our results suggest that the Dirac and strong Gamma kernels are more conducive of { oscillations} in the mean field firing rates, while the weak Gamma kernel promotes steady firing rates. Moreover, our numerical simulations also revealed cascades of bifurcations between periodic and quasi-periodic oscillations in conjunction with a discrete time delay (see Fig. \ref{fig.sim.dirac.and.Sgamma.pp}), which could not be observed in the case of strong Gamma kernels with the same system parameters. 
 This suggests that the limiting case of discretely distributed delays, while likely the most delicate in terms of a physiological implementation, may also offer the optimal physiological substrate for the complex { aperiodic oscillations} associated to task-specific functions in certain brain areas. To further investigate this possibility, we considered a specific application of this theoretical framework to a simplified model of the basal ganglia circuit involved in Parkinson's Disease.

\begin{figure}[http]
\centering
\begin{minipage}[c]{0.48\linewidth}
	\centering
	\includegraphics[width=\linewidth]{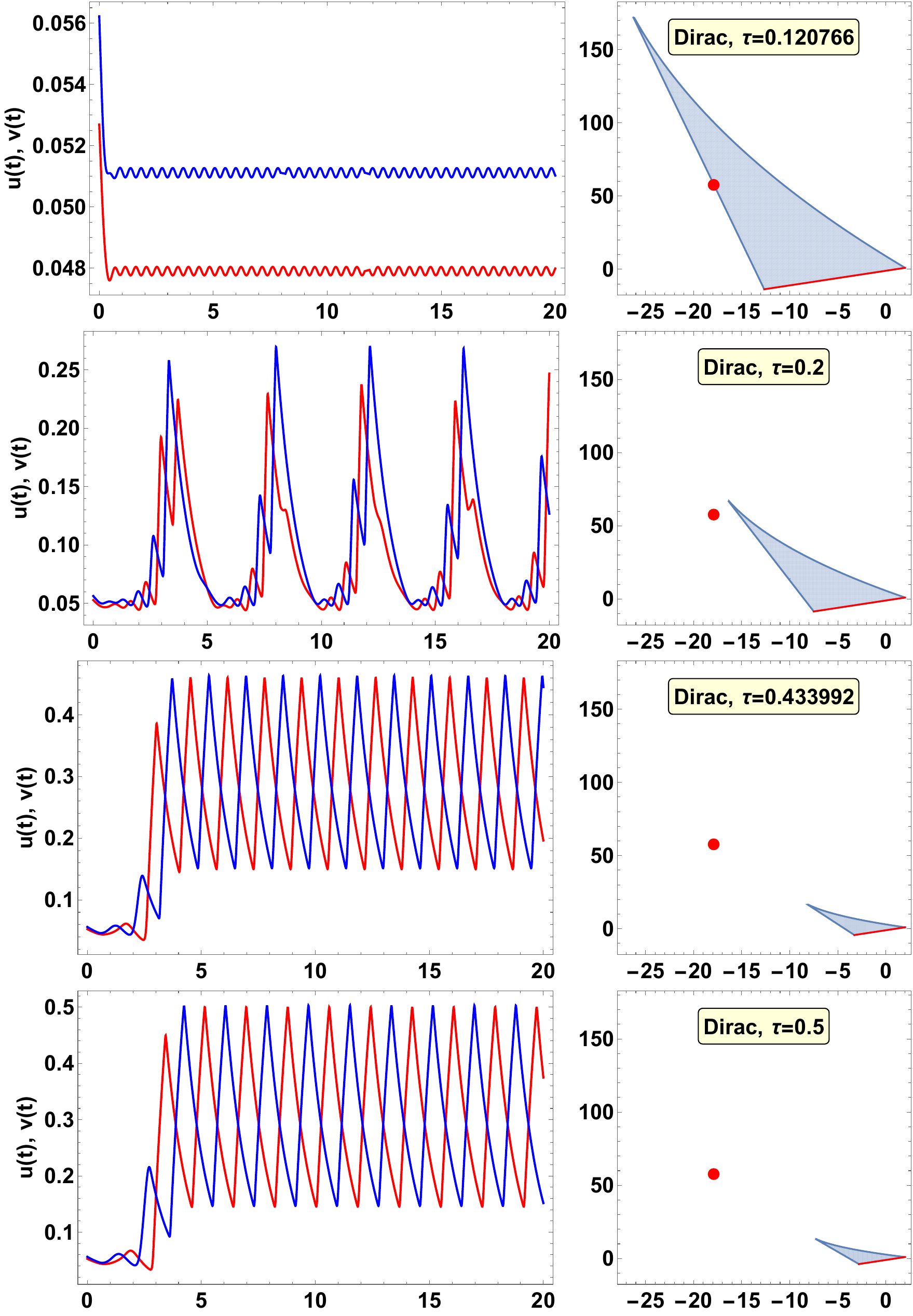}
\end{minipage}
\hspace*{0.02\linewidth}
\begin{minipage}[c]{0.48\linewidth}
\centering
  	\includegraphics[width=\linewidth]{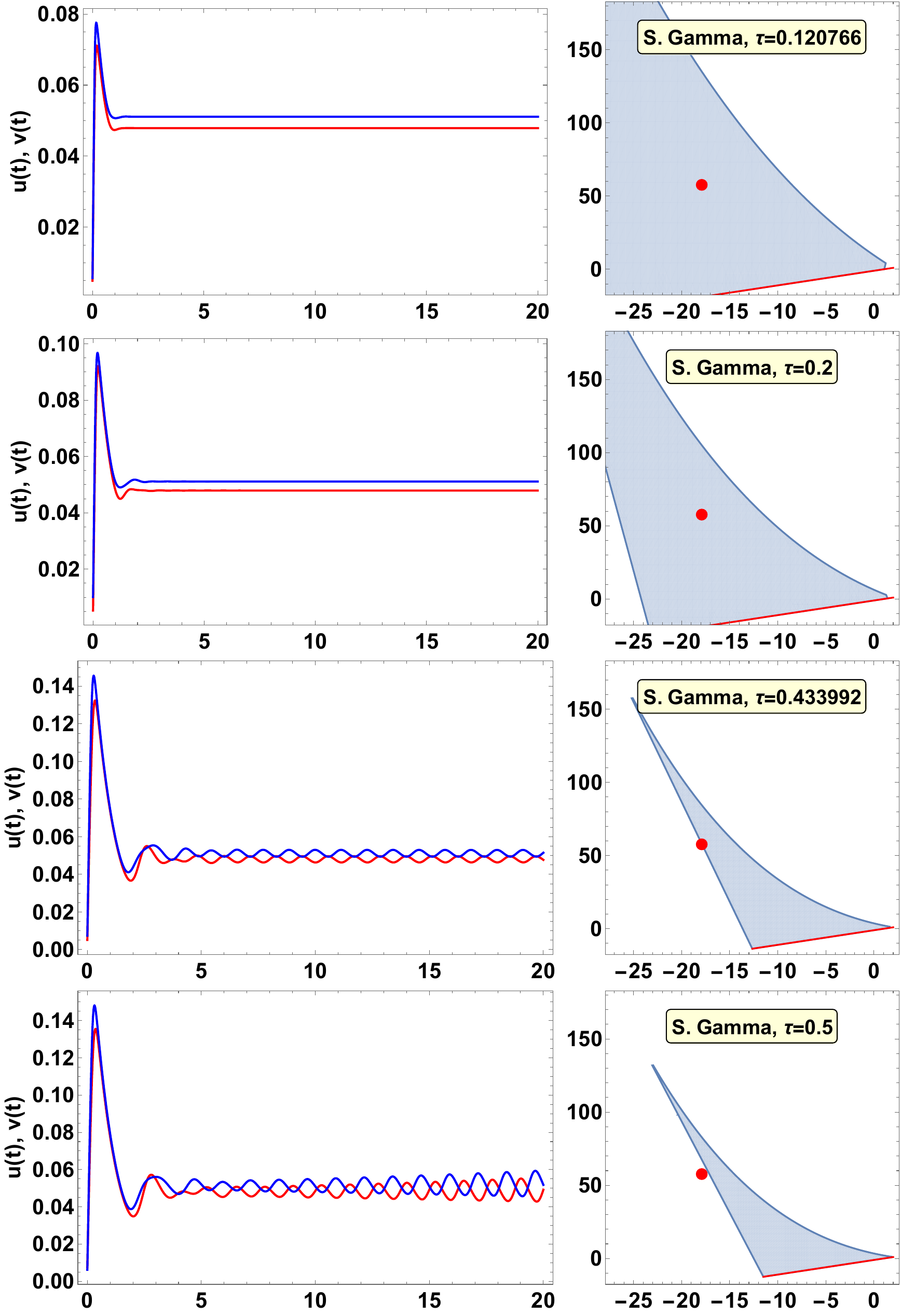}
\end{minipage}
	\caption{Evolution of state variables $u(t)$ and $v(t)$ of system \eqref{sys.wilson.cowan.dd} with discrete time-delay (left) and strong Gamma kernel (right) for values of the average time delay $\tau\in\{\tau_0^*,0.2,\tau_2^*,0.47\}$ (top to bottom). The values of the parameters are fixed: $\theta_u=0.1$, $\theta_v=0.2$, $a=d=-19$, $b=c=10$ and $\delta=10$. In this case, the equilibrium is $(u^\star,v^\star)=(0.0478985, 0.0511112)$ and the characteristic parameters are $\alpha=-17.8796$ and $\beta=57.7268$.  The same initial condition has been chosen in a neighborhood of the equilibrium. The critical value of $\tau$ in the Dirac kernel case is $\tau_0^* = 0.120766$ and in the strong Gamma kernel case is $\tau_2^* = 0.433992$. { The relative position of the point $(\alpha,\beta)$ with respect to the corresponding stability region is indicated.}}
	\label{fig.simulations}
\end{figure}

\begin{figure}[http]
	\centering
	\includegraphics[width=\linewidth]{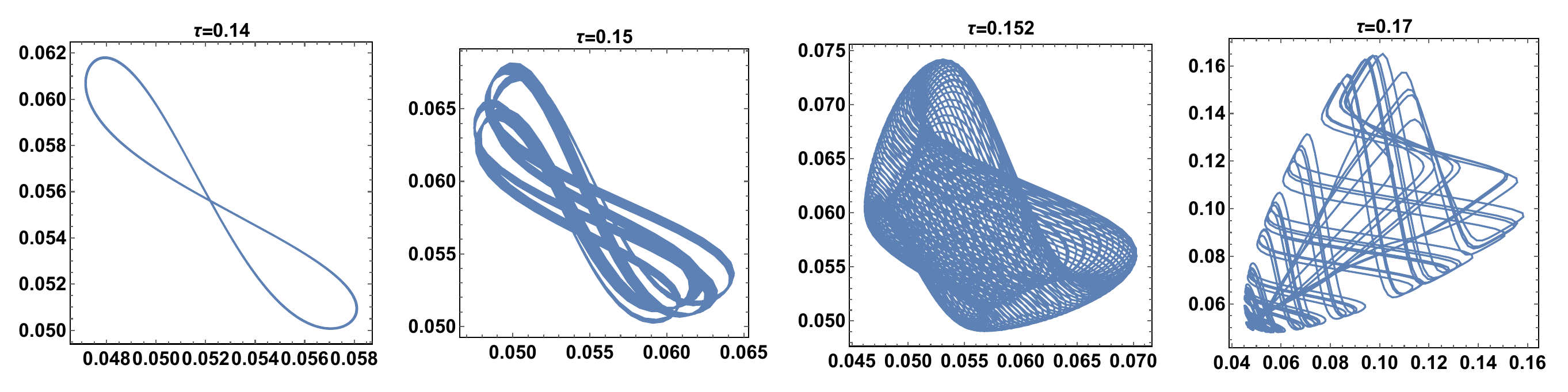}
	\caption{Periodic, quasi-periodic and chaotic orbits shown in the $(u,v)$-phase-plane for the Wilson-Cowan model with discrete time-delay, obtained for different values of $\tau$.}
	\label{fig.sim.dirac.and.Sgamma.pp}
\end{figure}

\subsection{Applications to the Parkinson's Disease}
\label{application2}

{ We will now use this theoretical layout to contextualize an existing computational model (based upon anatomical and electrophysiological research). This was developed by Holgado et al.~\cite{holgado2010} to describe an application to the subthalamic nucleus (STN) - globus pallidus (GPe) network involved in Parkinson's Disease:}
\begin{equation}\label{sys.holgado}
\left\{\begin{array}{l}
\ds\tau_S \dot{STN}(t)=-STN(t)+F_S\left[-w_{GS}GP(t-\Delta t_{GS}) + w_{CS} Ctx\right]\\
\\
\ds\tau_G \dot{GP}(t)=-GP(t)+F_G\left[w_{SG}STN(t-\Delta t_{SG}) - w_{GG}GP(t-\Delta t_{GG}) - w_{XG}Str\right]\\
\end{array}\right.
\end{equation}
where $STN$ and $GP$ represent the firing rates of the subthalamic nucleus and globus pallidus, respectively; $Ctx$, $Str$ are cortex, and striatum constant inputs; $\tau_S$, $\tau_G$ are time constants associated with the $STN$, $GP$ populations; $\Delta t_{MN}$ are trasmission delays from population $M$ to population $N$ and $F_S(\cdot)$ and $F_G(\cdot)$ are activation functions for each population. The parameter values used in \cite{holgado2010} are given in Table \ref{synaptic.weights}. The synaptic connection weight parameters ($w_{SS}, w_{GS}, w_{SG}, w_{GG}$) have been  estimated both in a healthy state and a parkinsonian/diseased state, and are also given in Table \ref{synaptic.weights}. Moreover, for simplicity, equal time constants and time delays are considered (i.e. $\tau_S = \tau_G = 6$ (ms), $\Delta t = \Delta t_{GS} = \Delta t_{SG} = \Delta t_{GG}$ (ms)). { In our further investigation of this application, we will also assume equal transmission delays, although our theoretical setup is otherwise more general (for example, it transcends the need to linearize the activation functions, which was a building step in the original analysis in the reference).}

We consider a time rescaling $t\mapsto \tau_S t$, the rescaled state variables $u(t) = STN(\tau_S t)$, $v(t) = GP(\tau_S t)$, and the parameters $a=w_{SS}$, $b=-w_{GS}$, $c=w_{SG}$, $d=-w_{GG}$, $\theta_u = w_{CS}Ctx$, $\theta_v=-w_{XG}Str$. The activation functions are as in \cite{holgado2010}: \[f_1(x) = F_S(x) =\ds \frac{M_S B_S}{B_S+exp(-4x/M_S)(M_S - B_S)}\quad\text{and}\quad f_2(x) = F_G(x) =\ds \frac{M_G B_G}{B_G+exp(-4x/M_G)(M_G - B_G)}\]
The state variables $u$ and $v$ verify a system of the type \eqref{sys.wilson.cowan.dd} with an average time delay $\tau=\Delta t/\tau_S$ (taking into account that in the original model from \cite{holgado2010}, discrete time delays are considered).

\setlength{\arrayrulewidth}{0.5mm}
\setlength{\tabcolsep}{15pt}
\begin{table}[http]
    \centering
\begin{tabular}{|c|c|c|c|c|}
\hline
& & & &\\
    Parameter & Healthy state & Diseased state & Other Parameters & Value \\
& & & &\\
\hline
    $w_{SS}$ & $0$ & $0$ & $Ctx$ & $27$ spk/s\\
\hline
    $w_{SG}$ & $19.0$ & $20.0$ & $Str$ & $2$ spk/s\\
\hline
    $w_{GS}$ & $1.12$ & $10.7$ & $M_S$ & $300$ spk/s\\
\hline
    $w_{GG}$ & $6.60$ & $12.3$ & $B_S$ & $17$ spk/s\\
\hline
    $w_{CS}$ & $2.42$ & $9.2$ & $M_G$ & $400$ spk/s\\
\hline
    $w_{XG}$ & $15.1$ & $139.4$ & $B_G$ & $75$ spk/s\\
\hline
\end{tabular}
\caption{Synaptic connection weight parameters and other parameters estimated in \cite{holgado2010}}
\label{synaptic.weights}
\end{table}

{ For the healthy state parameters given in Table \ref{synaptic.weights}, the characteristic parameter values are $(\alpha_H,\beta_H)=(-3.06805,2.24878)$. As in this case, $\alpha_H^2>4\beta_H$, a Hopf bifurcation may occur only in the conditions of Theorem \ref{thm.stab.dom.bounded}, when $(\alpha_H,\beta_h)$ belong to the line segment $(l_\tau)$. In the Dirac kernel case, the critical value of $\tau=\Delta t/\tau_S$ responsible for the occurrence of a Hopf bifurcation is computed as described in the previous section: $\tau_{0,H}^*=1.367$, and the frequency of oscillations occurring in system \eqref{sys.holgado} near the bifurcation point is $\omega^*/(2\pi\tau_S\tau^*)=41.5133~Hz$, which agrees with the numerical simulations presented in Fig. \ref{fig.simulations.dirac}. On the other hand, in the case of a strong Gamma kernel, it may be easily checked that $(\alpha_H,\beta_H)$ belongs to the "smallest" stability region $S_2(\alpha,\beta)$ indicated in Example \ref{ex.strong.gamma}, and therefore the equilibrium of system \eqref{sys.holgado} is asymptotically stable for any $\tau$ (see Fig. \ref{fig.simulations.strong}). Likewise, in the case of a weak Gamma kernel, $(\alpha_H,\beta_H)$ belongs to the "smallest" stability region $S_1(\alpha,\beta)$ indicated in Example \ref{ex.weak.gamma}, and hence the equilibrium of system \eqref{sys.holgado} is asymptotically stable for any $\tau$, as shown in Fig. \ref{fig.simulations.weak}.

For the pathological state parameters given in Table \ref{synaptic.weights}, the characteristic parameter values are $(\alpha_P,\beta_P)=(-2.53928,11.2213)$. This time, as $\alpha_P^2<4\beta_P$, a Hopf bifurcation occurs only when $(\alpha_P,\beta_P)\in (\gamma_\tau)$. To find the critical value of $\tau$ for a particular delay kernel, we numerically solve for smallest positive $\omega$ and $\tau$ the system which provides the parametric equations of the curve $(\gamma_\tau)$, as in Theorem \ref{thm.stab.dom.bounded} or Theorem \ref{thm.stab.dom.unbounded}, and we obtain the following values for the considered kernels: 
\begin{itemize}
\item Dirac: $\tau_{0,P}^*=0.216411$ and corresponding frequency of oscillations in \eqref{sys.holgado}: $84.8049~Hz$;
\item weak Gamma: 
$\tau_{1,P}^*=0.619418$ and corresponding frequency of oscillations in \eqref{sys.holgado}: $50.7756~Hz$;
\item strong Gamma: $\tau_{2,P}^*=0.283222$ and corresponding frequency of oscillations in \eqref{sys.holgado}: $72.5652~Hz$;
\end{itemize}
Again, the frequencies have been computed according to the formula $\omega^*/(2\pi\tau_S\tau^*)$, due to the time rescaling which has been used. The theoretically predicted frequencies correspond to the numerically computed values, as seen in Figs. \ref{fig.simulations.dirac}, \ref{fig.simulations.strong} and \ref{fig.simulations.weak}. 
}

We simulated the behaviour of the system \eqref{sys.holgado}, with the above chosen parameters and activation functions. The behaviors found illustrate in a clinical context our theoretical results obtained in this paper (e.g. Theorems \ref{thm.stab.dom.bounded} and \ref{thm.stab.dom.unbounded}) and expand on the modeling results in the reference \cite{holgado2010}. 

As a first step, Figures~\ref{fig.simulations.dirac}, \ref{fig.simulations.strong} and \ref{fig.simulations.weak} illustrate the behavior of each of the two state variables $u$ and $v$ for different delay kernels: for discrete delays (Figure~\ref{fig.simulations.dirac}) for strong Gamma delays (Figure~\ref{fig.simulations.strong}) and for weak Gamma delays (Figure~\ref{fig.simulations.weak}). In each figure, the left panels represent the behavior of the system with coupling parameters within a healthy functioning regime, while the right panels illustrate the behavior for coupling parameters within the PD range (as shown in Table~\ref{synaptic.weights}). The different rows of the figures illustrate different values of the delay $\Delta t/\tau_S \in [0,2]$, interval compatible with the physiological range (as per the table in the reference~\cite{holgado2010}). Notice that, in the PD case (right panels), sustained oscillations in the Beta range can be readily triggered by delay values within this range ($\Delta t/\tau_S \sim 0.21$ in the case of the discrete kernel, $\Delta t/\tau_S \sim 0.28$ in the case of the strong Gamma kernel, $\Delta t/\tau_S \sim 0.61$ in the case of the weak Gamma kernel). Interestingly, while the oscillation amplitudes increase with the delay (top to bottom panels), the frequency is higher for shorter delays. 

In contrast, note that much higher delay values are required in order to trigger stable oscillations in the healthy regime (they appear at a critical value of $\Delta t/\tau_S \sim 1.36$ for the discrete kernel, and they never occur in the case of weak or strong Gamma kernel). This is in agreement with the idea that { periodic oscillations in the Beta range are greatly enhanced in PD}, as proposed in the original reference~\cite{holgado2010}. { Empirical studies support the link between increased Beta oscillatory activity and PD symptoms (such as  bradykinesia and rigidity)~\cite{little2014functional}, as discussed in more detail in the last section.}

\begin{figure}[http]
\centering
\begin{minipage}[c]{0.47\linewidth}
	\centering
	\includegraphics[width=\linewidth]{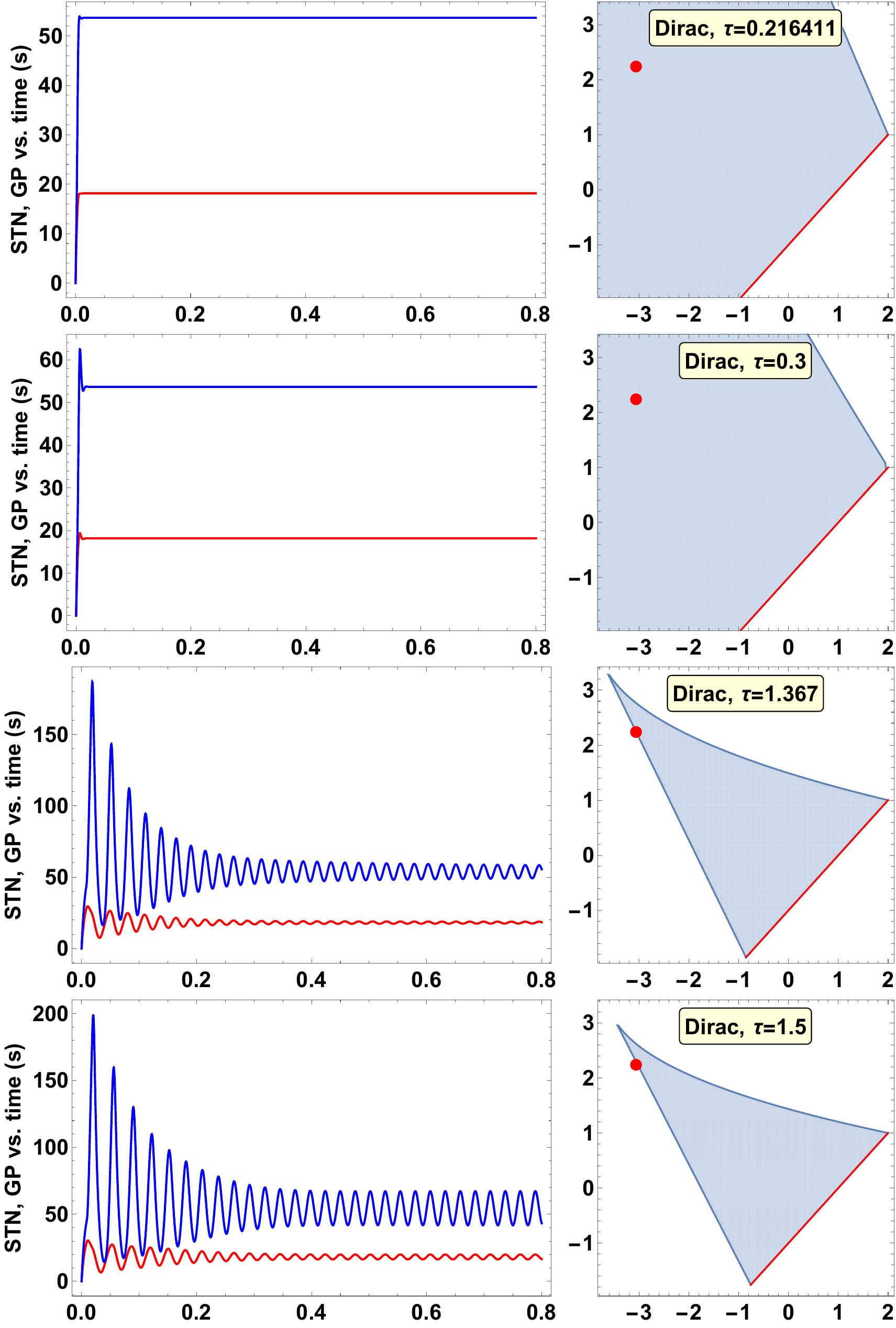}
\end{minipage}
\hspace*{0.02\linewidth}
\begin{minipage}[c]{0.47\linewidth}
\centering
  	\includegraphics[width=\linewidth]{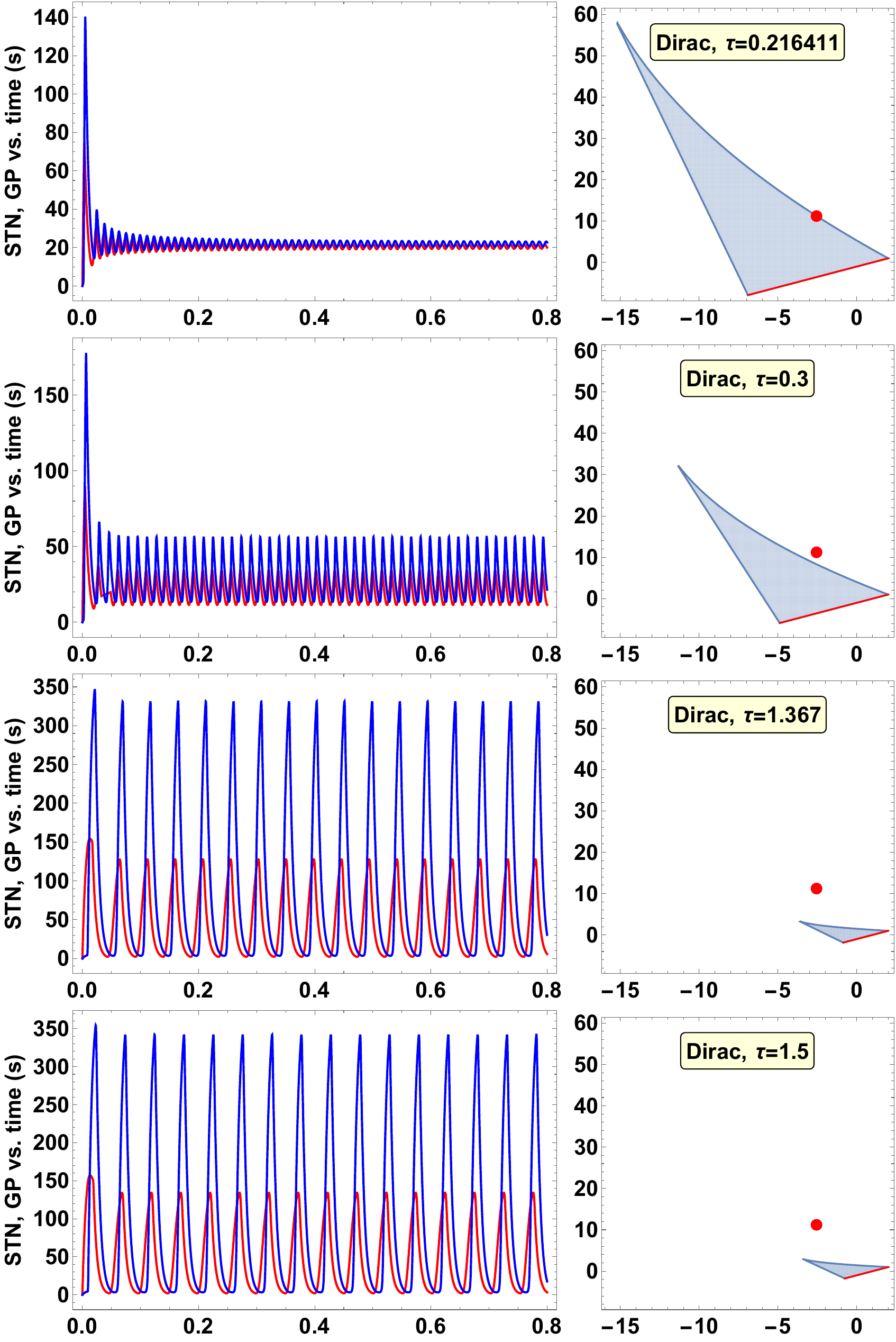}
\end{minipage}
	\caption{Evolution of state variables $STN(t)$ and $GP(t)$ of system \eqref{sys.holgado} with discrete time-delay with healthy state parameters (left) and diseased (parkinsonian) state parameters (right) as given in Table \ref{synaptic.weights} for the values of  $\tau=\Delta t/\tau_S \in\{\tau_{0,P}^*,0.3,\tau_{0,H}^*,1.4\}$ (top to bottom). The critical values of $\tau=\Delta t/\tau_S$ are: $\tau_{0,H}^*=1.367$ in the healthy state and  $\tau_{0,P}^*=0.216411$ in the diseased state. The position of the points $(\alpha_H,\beta_H)=(-3.06805,2.24878)$ and $(\alpha_P,\beta_P)=(-2.53928,11.2213)$ is indicated with respect to the corresponding stability region.}
	\label{fig.simulations.dirac}
\end{figure}

\begin{figure}[http]
\centering
\begin{minipage}[c]{0.47\linewidth}
	\centering
	\includegraphics[width=\linewidth]{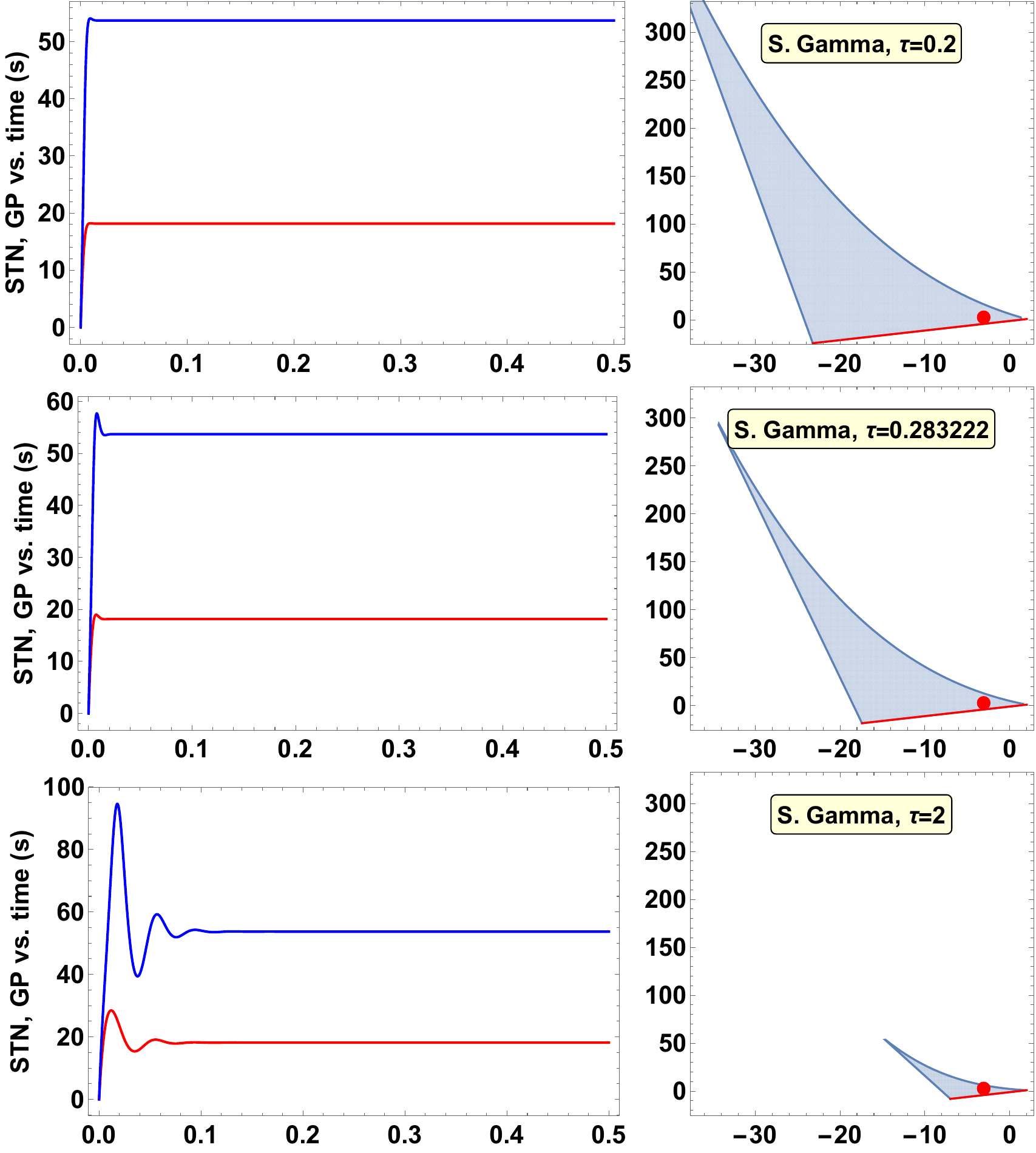}
\end{minipage}
\hspace*{0.02\linewidth}
\begin{minipage}[c]{0.48\linewidth}
\centering
  	\includegraphics[width=\linewidth]{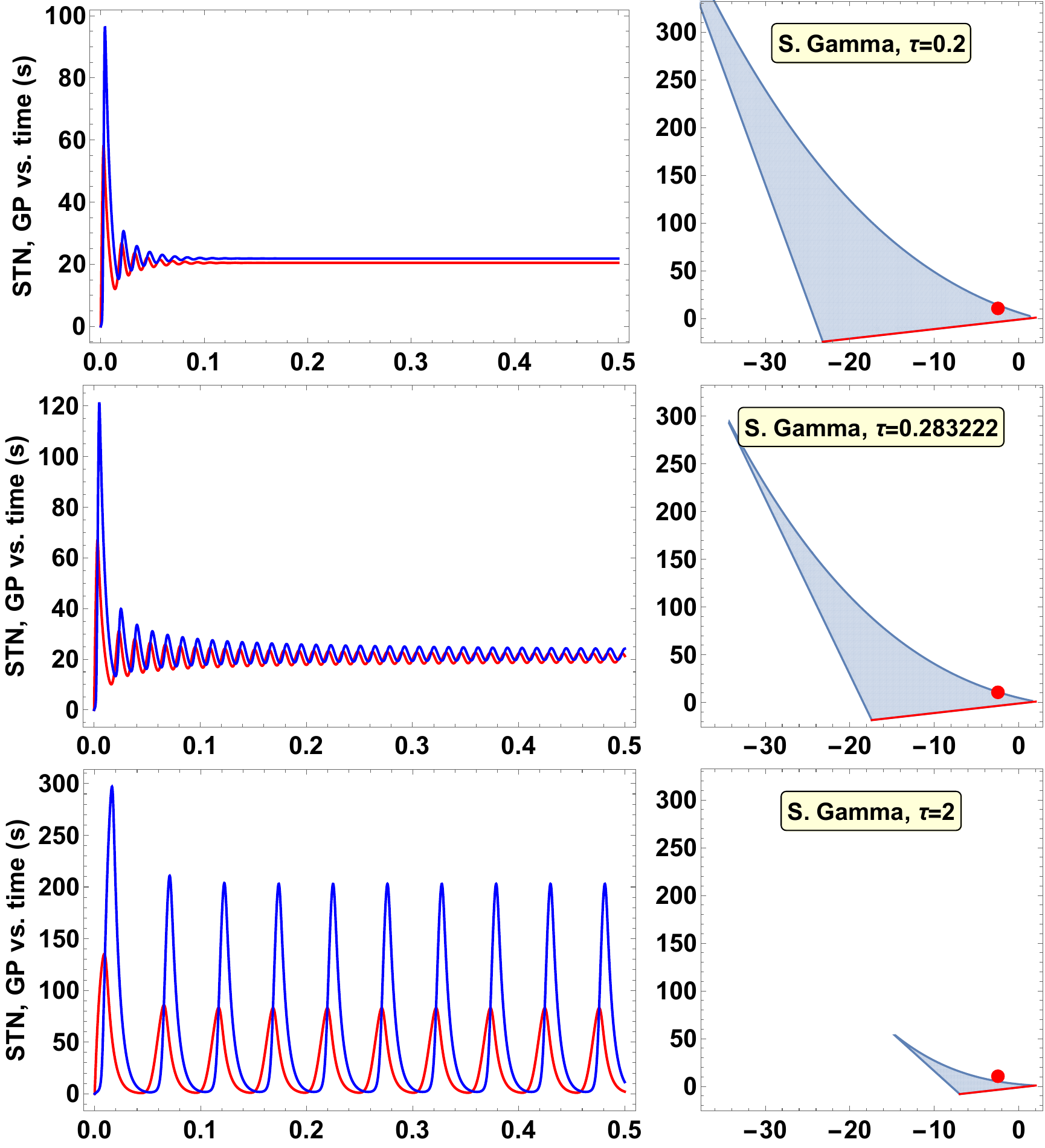}
\end{minipage}
	\caption{Evolution of state variables $STN(t)$ and $GP(t)$ of system \eqref{sys.holgado} with strong Gamma delay kernels with healthy state parameters (left) and diseased (parkinsonian) state parameters (right) as given in Table \ref{synaptic.weights} for the values of $\tau=\Delta t/\tau_S \in\{0.2,\tau_{2,P}^*,2\}$ (top to bottom). In the healthy case, the equilibrium is always asymptotically stable, while in the diseased case, the critical value of $\tau=\Delta t/\tau_S$ is $\tau_{2,P}^*=0.283222$. The position of the points $(\alpha_H,\beta_H)=(-3.06805,2.24878)$ and $(\alpha_P,\beta_P)=(-2.53928,11.2213)$ is indicated with respect to the corresponding stability region.	}
	\label{fig.simulations.strong}
\end{figure}

\begin{figure}[http]
\centering
\begin{minipage}[c]{0.47\linewidth}
	\centering
	\includegraphics[width=\linewidth]{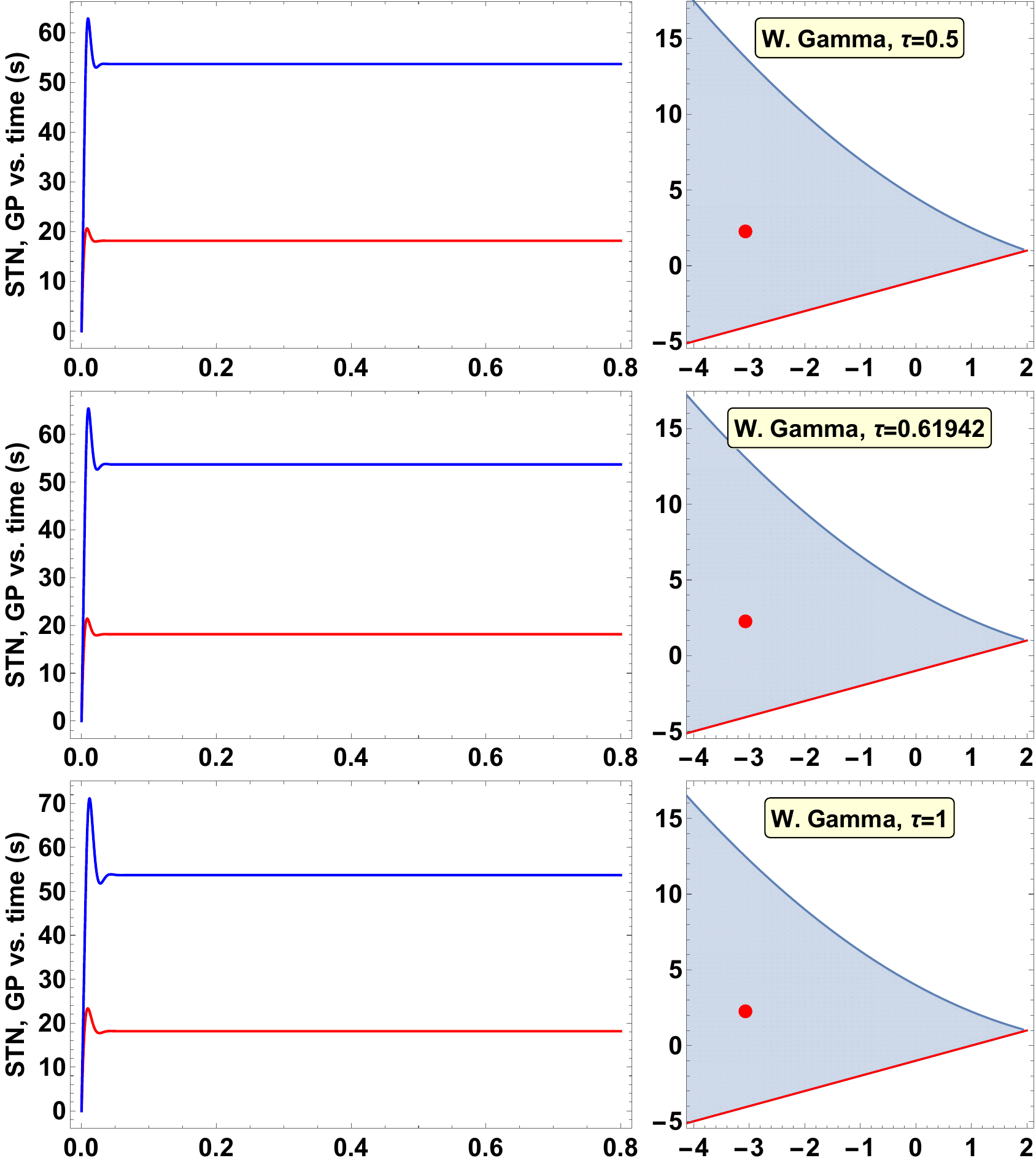}
\end{minipage}
\hspace*{0.02\linewidth}
\begin{minipage}[c]{0.48\linewidth}
\centering
  	\includegraphics[width=\linewidth]{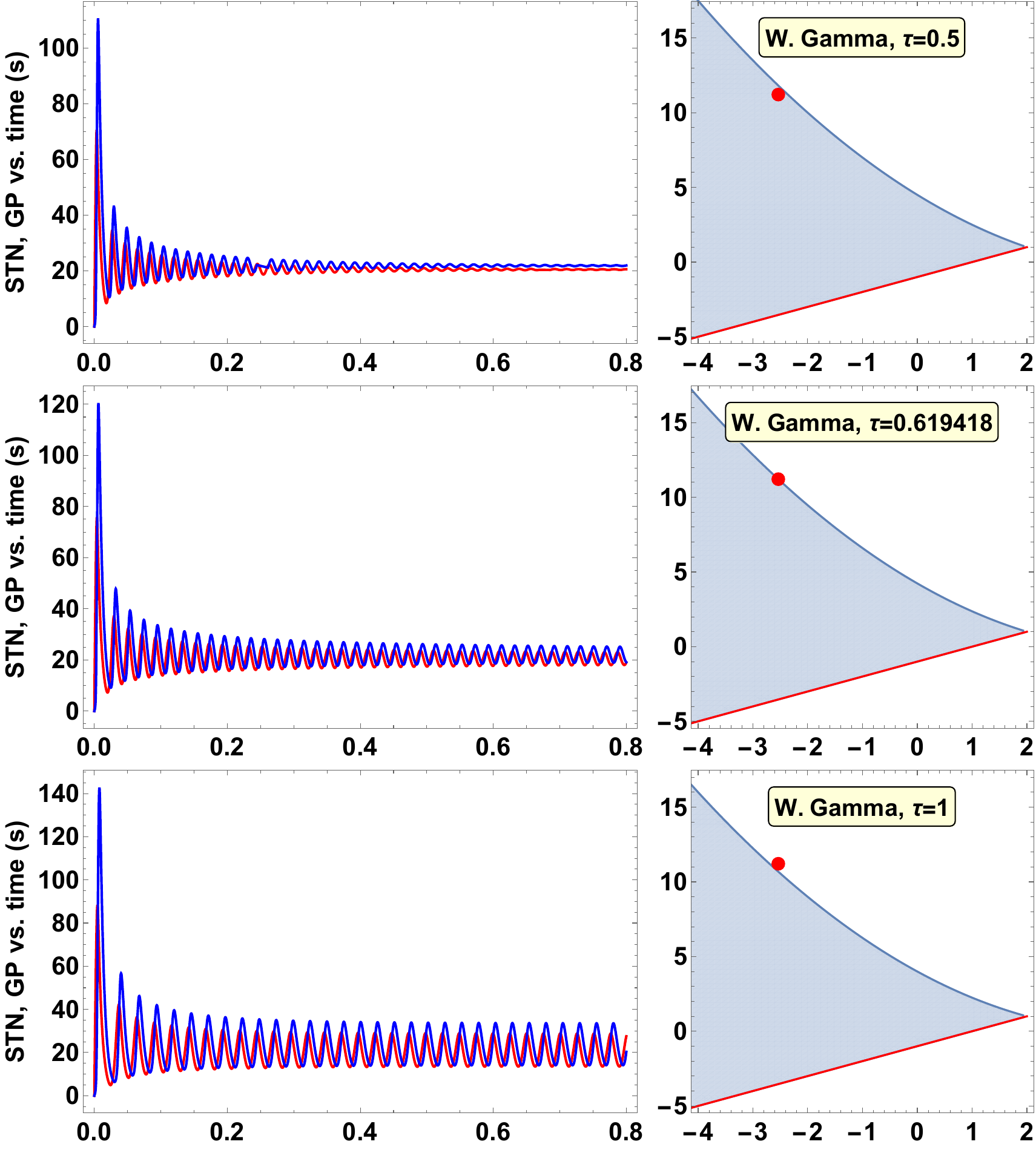}
\end{minipage}
	\caption{Evolution of state variables $u(t)$ and $v(t)$ of system \eqref{sys.wilson.cowan.dd} with weak Gamma delay kernels with healthy state parameters (left) and diseased (parkinsonian) state parameters (right) as given in Table \ref{synaptic.weights} for the values of $\tau=\Delta t/\tau_S \in\{0.5,\tau_{1,P}^*,1\}$ (top to bottom). In the healthy case, the equilibrium is always asymptotically stable, while in the diseased case, the critical value of $\tau=\Delta t/\tau_S$ is $\tau_{1,P}^*=0.619418$. The position of the points $(\alpha_H,\beta_H)=(-3.06805,2.24878)$ and $(\alpha_P,\beta_P)=(-2.53928,11.2213)$ is indicated with respect to the corresponding stability region.}
	\label{fig.simulations.weak}
\end{figure}

The differential propensity of the system for { oscillations} in the PD versus the healthy regime is further illustrated in Figures~\ref{fig:critical.ratio.dirac},~\ref{fig:critical.ratio.strong1}, \ref{fig:critical.ratio.strong2}, \ref{fig:critical.ratio.weak1} and~\ref{fig:critical.ratio.weak2}. These deliver a more comprehensive description of the subtle interplay (between coupling strengths, delay kernel and delay value) that orchestrates the triggering and maintenance of { oscillatory behavior.}

Figure~\ref{fig:critical.ratio.dirac} illustrates differences between the onset of { oscillations} in the PD system versus the healthy system, in the case of discrete delays. For each panel, the coupling parameters $w_{SS}$, $w_{GG}$, $w_{CS}$ and $w_{XG}$ were fixed (not shown): to the healthy state values in Table~\ref{synaptic.weights} (in the left panel) and to the PD values (in the right panel). Then, for each pair of cross-connectivity parameters $w_{SG}$ and $w_{GS}$ in the appropriate range, the critical value (i.e., the value corresponding to onset of oscillations) of $\Delta t/\tau_S$ was computed. The corresponding value was then plotted at each parameter point $(w_{SG},w_{GS})$ as an associated color in a blue to red color map (color bar shown on the right of each parameter plot). 

\color{black}

\begin{figure}[http]
\centering
\begin{minipage}[c]{0.48\linewidth}
	\centering
	\includegraphics[width=\linewidth]{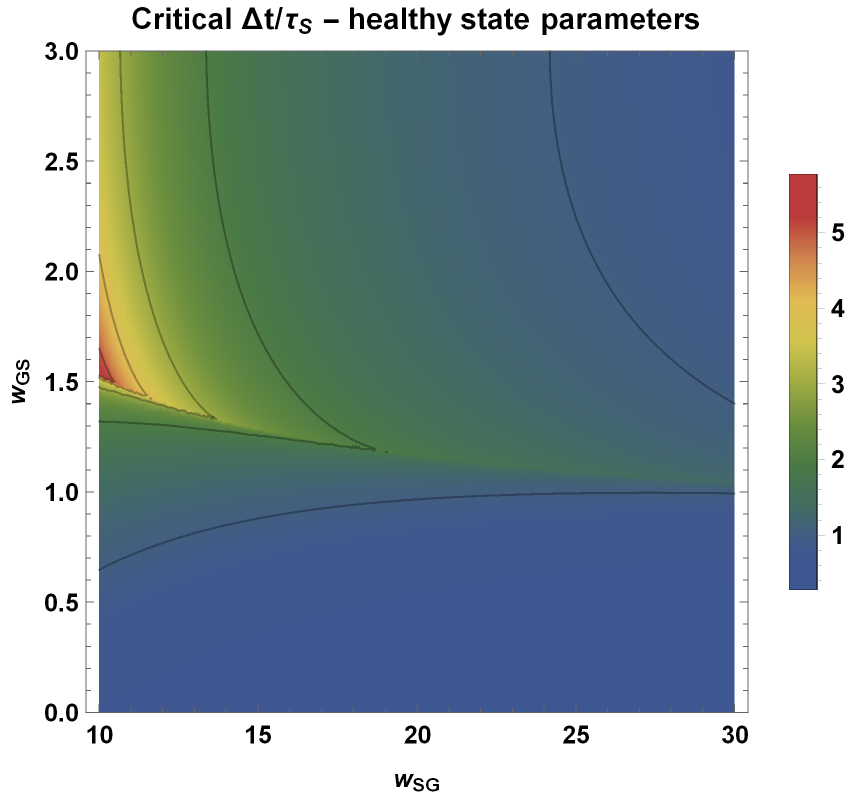}
\end{minipage}
\hspace*{0.02\linewidth}
\begin{minipage}[c]{0.48\linewidth}
\centering
  	\includegraphics[width=\linewidth]{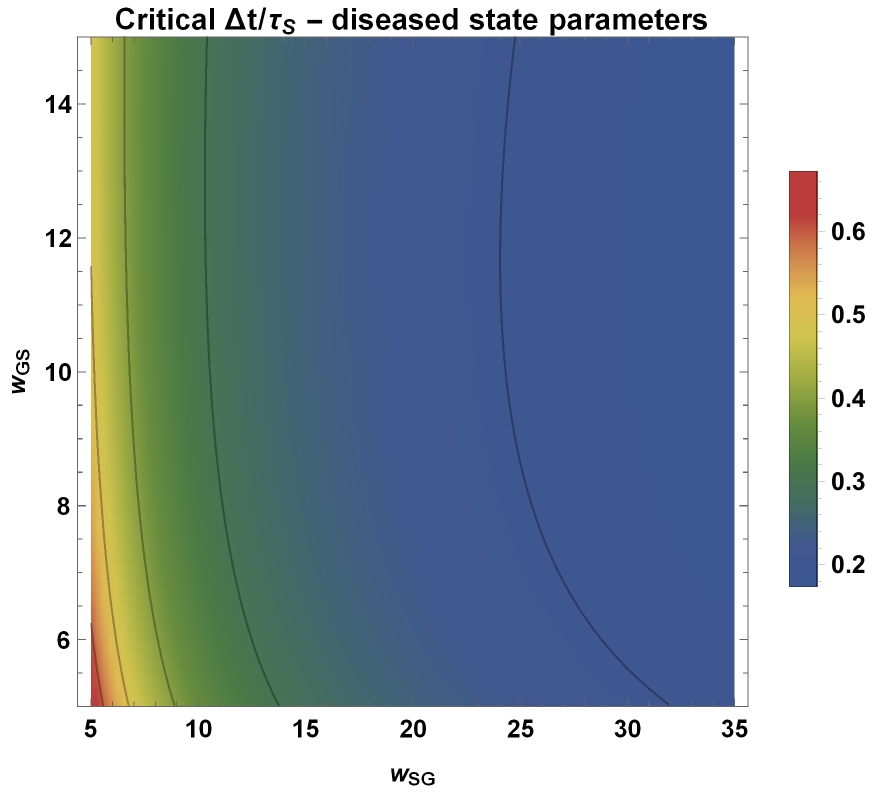}
\end{minipage}
    \caption{Discrete time-delay. Critical values of the $\Delta t/\tau_S$ ratio which mark the onset of oscillatory behavior for the two cases: healthy (left) and parkinsonian (right). All parameter values have been fixed to those given in Table 1, apart from $w_{SG}$ and $w_{GS}$.}
    \label{fig:critical.ratio.dirac}
\end{figure}

An immediate comparison between the two figure panels reiterates (in this broader parameter frame) the idea that, for Dirac distributed delays, sustained oscillations are accessible in the PD coupling range at significantly shorter delay times (consistent with the physiological values) than in the healthy coupling range. A more in depth analysis of the color patterns may suggest candidate mechanisms, based on synaptic remodeling, for onset and cessation of 
{ oscillatory} behavior in each case. For example, in healthy systems, a weaker GP to STN coupling is more likely to result in sustained { oscillations} (since it lowers the critical delay to values within the physiological range). In PD systems, operating with short delay times ($\Delta t/\tau_S = 0.2-0.3$), strengthening of the STN to GP synaptic coupling may readily trigger sustained { oscillations} in PD patients, while weakening it can lead to ceasing the { oscillatory} behavior. This suggests the importance of having access to information not only on the synaptic coupling profile, but also on knowledge of the type of delay range and distribution involved in the neural integration, in the circuit in question.

Continuing along with this thought, the rest of this section revisits, in the context of this application, the importance of the subtle, but crucial differences in dynamics induced by different profiles in the delay distribution. Figures~\ref{fig:critical.ratio.strong1} and ~\ref{fig:critical.ratio.strong2} illustrate and compare the same exact coupling ranges as in Figure~\ref{fig:critical.ratio.dirac} (for the healthy state on the left, and for the PD state on the right), for the case of the strong Gamma kernel. Notice that, in the pathological case, the critical values of $\Delta t/\tau_S$ are sightly higher throughout the parameter plane than the similar values in the case of the Dirac kernel. However, the candidate mechanisms proposed for the trigger and stop of { oscillations} in the case of the Dirac kernel may equally apply. 

On the other hand, by contrast with the Dirac case, { oscillations are} no longer possible in the healthy range of coupling parameters in the case of the strong Gamma kernel. Theoretically, this is because the healthy coupling range (as shown in the left hand panel of Figure~\ref{fig:critical.ratio.dirac} and re-expressed as the red subset in the $(\alpha,\beta)$ domain, in Figure~\ref{fig:critical.ratio.strong1}) is fully contained in this case in the stability region of the system (shown as a grey shaded area in the $(\alpha,\beta)$ domain in Figure~\ref{fig:critical.ratio.strong1}), as computed in Theorem 2.7. Thus, previously shown in Figure 6a, any tendency towards { oscillations} is quickly suppressed, and the system converges to a steady firing regime.

In line with our previous analysis, { stable oscillations are} gated by even longer delay values when using weak Gamma distributed delays. Indeed, Figures~\ref{fig.simulations.weak},~\ref{fig:critical.ratio.weak1} and~\ref{fig:critical.ratio.weak2} illustrate the behavior of the system using a weak Gamma kernel, with the same fixed parameters as for the Dirac and strong Gamma kernels. The healthy coupling range remains a subset of the stability domain (as shown in Figure~\ref{fig:critical.ratio.weak1}), hence { stable oscillations are} not possible as a long-term outcome in this scenario (Figure~\ref{fig.simulations.weak}a). For the pathological coupling range, oscillations are possible for values of $\Delta t/\tau_S$ larger than $\sim 0.5$, but only if both the STN-GP and GP-STN synaptic pathways are strong enough (Figure~\ref{fig:critical.ratio.weak2} shows that oscillations are possible in this case only in the upper right corner of the $(w_{SG},w_{GS})$ parameter region).

\begin{figure}[http]
\centering
\begin{minipage}[t]{0.48\linewidth}
	\centering
	\includegraphics[width=0.88\linewidth]{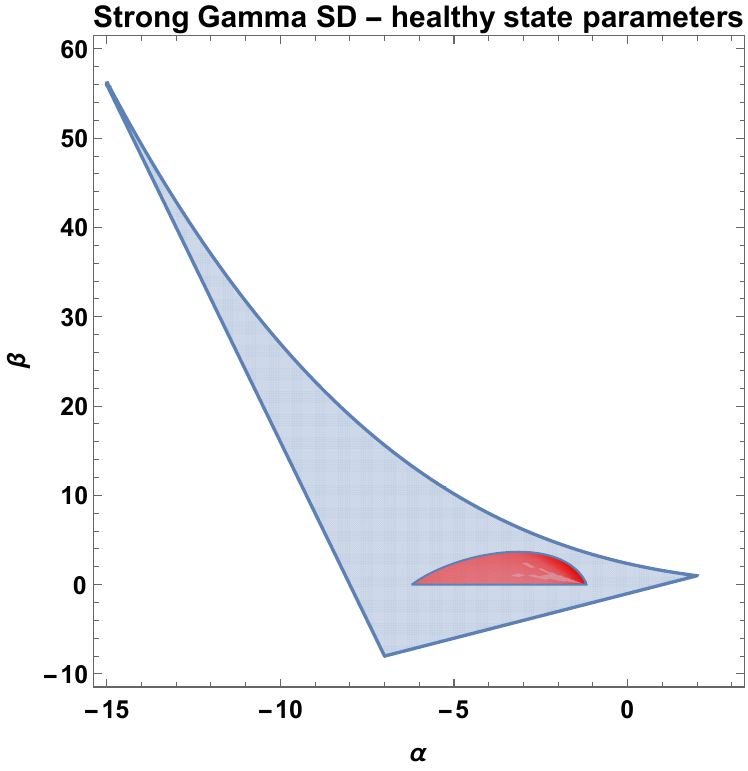}
	\caption{Strong Gamma delay kernel. Asymptotic stability for any $\tau>0$, for any values of the parameters $w_{GS}\in(0,20)$, $w_{SG}\in(0,30)$, when the other parameters and fixed at the healthy values given in Table 1.}
	\label{fig:critical.ratio.strong1}
\end{minipage}
\hspace*{0.02\linewidth}
\begin{minipage}[t]{0.48\linewidth}
\centering
  	\includegraphics[width=\linewidth]{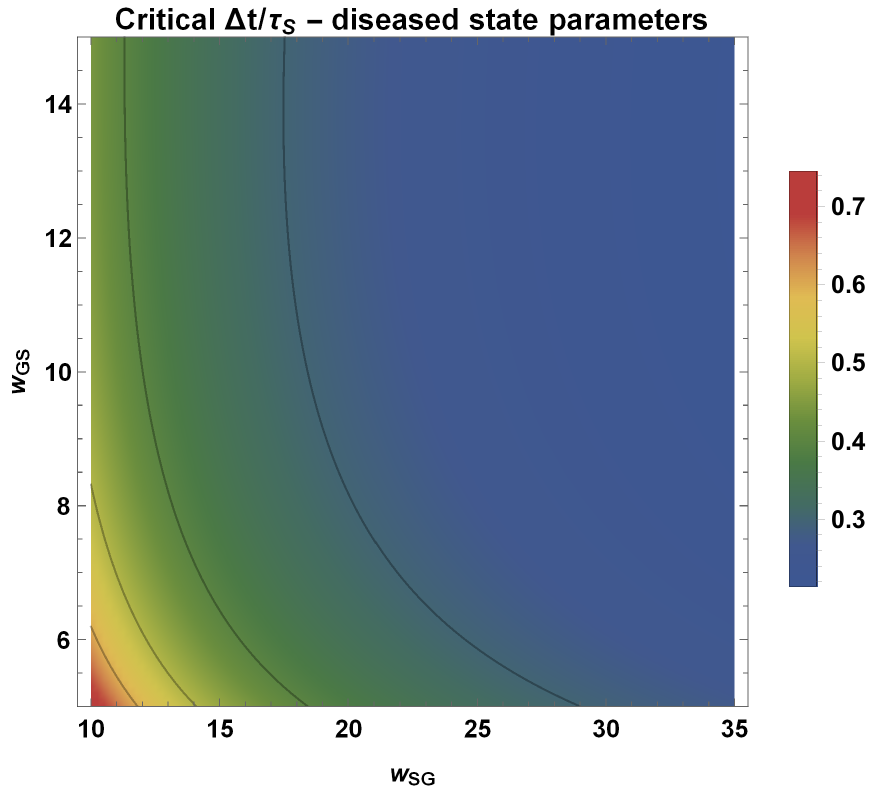}
  	 \caption{Strong Gamma delay kernel. Critical values of the $\Delta t/\tau_S$ ratio which mark the onset of oscillatory behavior for the parkinsonian case (with parameter values given by Table 1.)}
    \label{fig:critical.ratio.strong2}
\end{minipage}
\end{figure}

\begin{figure}[http]
\centering
\begin{minipage}[t]{0.48\linewidth}
	\centering
	\includegraphics[width=0.88\linewidth]{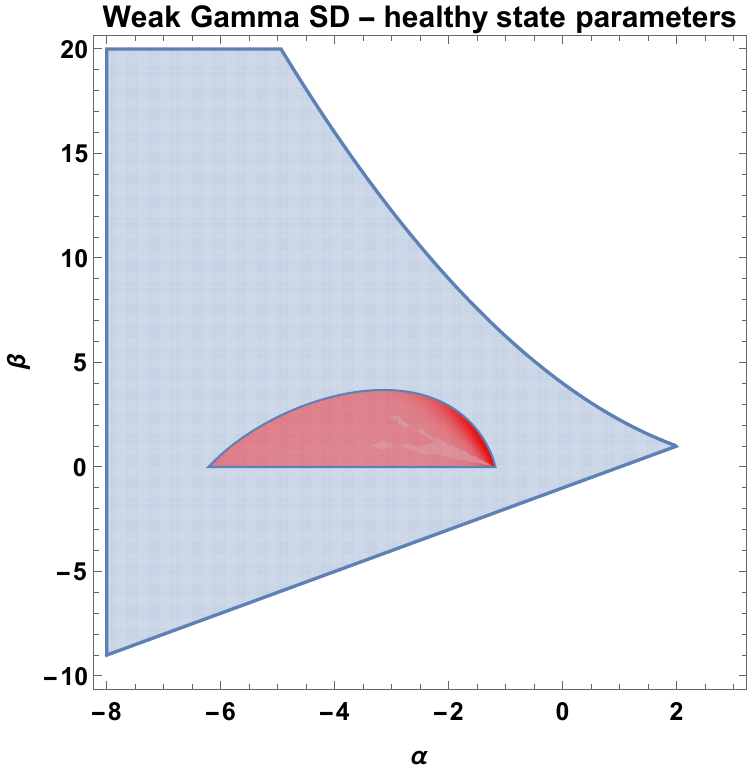}
	\caption{Weak Gamma delay kernel. Asymptotic stability for any $\tau>0$, for any values of the parameters $w_{GS}\in(0,20)$, $w_{SG}\in(0,30)$, when the other parameters and fixed at the healthy values given in Table 1.}
	\label{fig:critical.ratio.weak1}
\end{minipage}
\hspace*{0.02\linewidth}
\begin{minipage}[t]{0.48\linewidth}
\centering
  	\includegraphics[width=\linewidth]{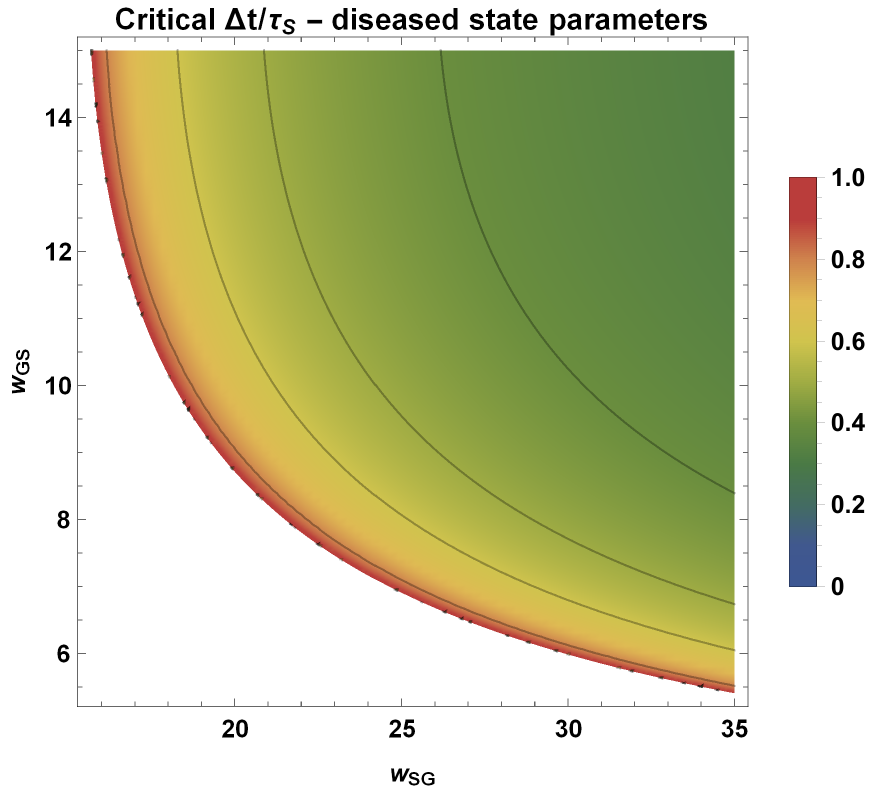}
  	 \caption{Weak Gamma delay kernel. Critical values of the $\Delta t/\tau_S$ ratio which mark the onset of oscillatory behavior for the parkinsonian case (with parameter values given by Table 1.)}
    \label{fig:critical.ratio.weak2}
\end{minipage}
\end{figure}

A further comparison between { oscillations} in the case of Dirac, weak and strong Gamma distributed delays is shown in Figure~\ref{fig.sim.strong.gamma.park.pp}. For the PD coupling range, each row of panels captures the emergence of oscillations when increasing the delay $\Delta t/\tau_S$ past the critical value: for the Dirac kernel (top row), for the weak Gamma kernel (middle row), and for the strong Gamma kernel (bottom row). As previously discussed, the critical delay for the Dirac distribution ($\Delta t/\tau_S \sim 0.21$) is slightly lower than that for the strong Gamma distribution ($\Delta t/\tau_S \sim 0.28$). In addition, the figure confirms that the critical value for the weak Gamma distribution is significantly larger than both ($\Delta t/\tau_S \sim 0.61$), as one would have expected from the theoretical results illustrated in Figure~\ref{fig.stab.dom}. The significance of these differences will be further discussed in the next section.

\color{black}

\begin{figure}[http]
	\centering
	\includegraphics[width=\linewidth]{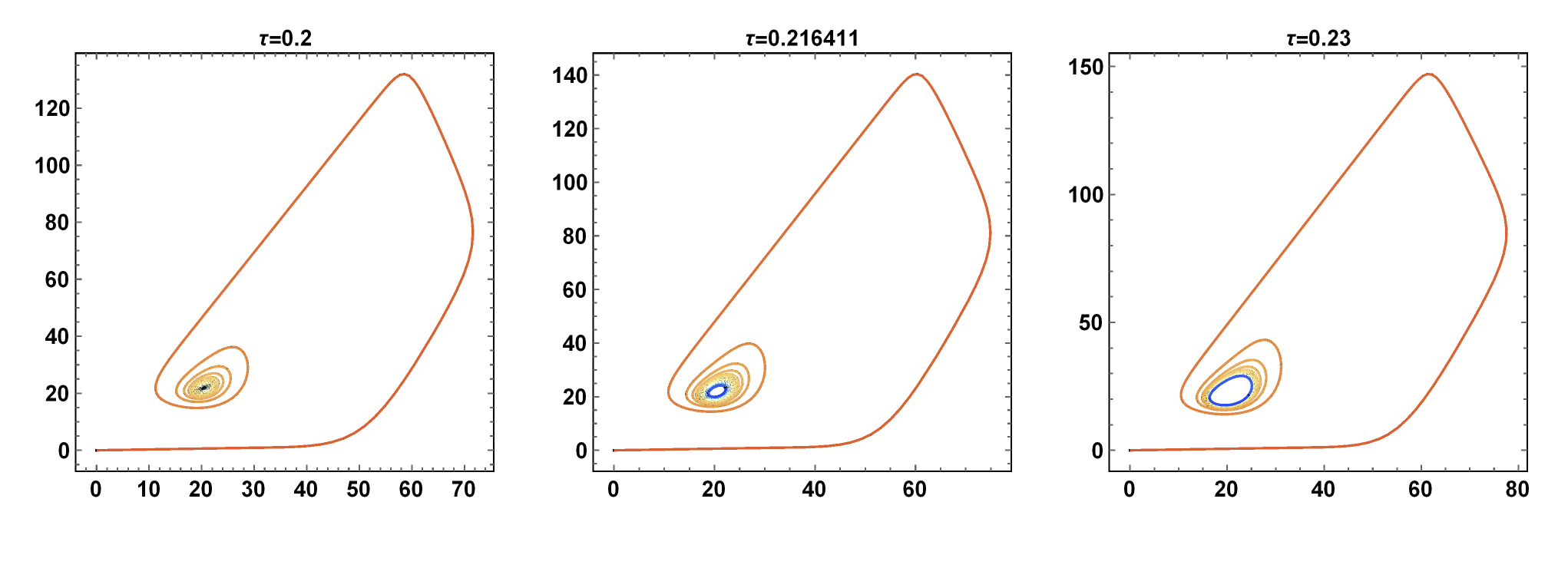}
	\includegraphics[width=\linewidth]{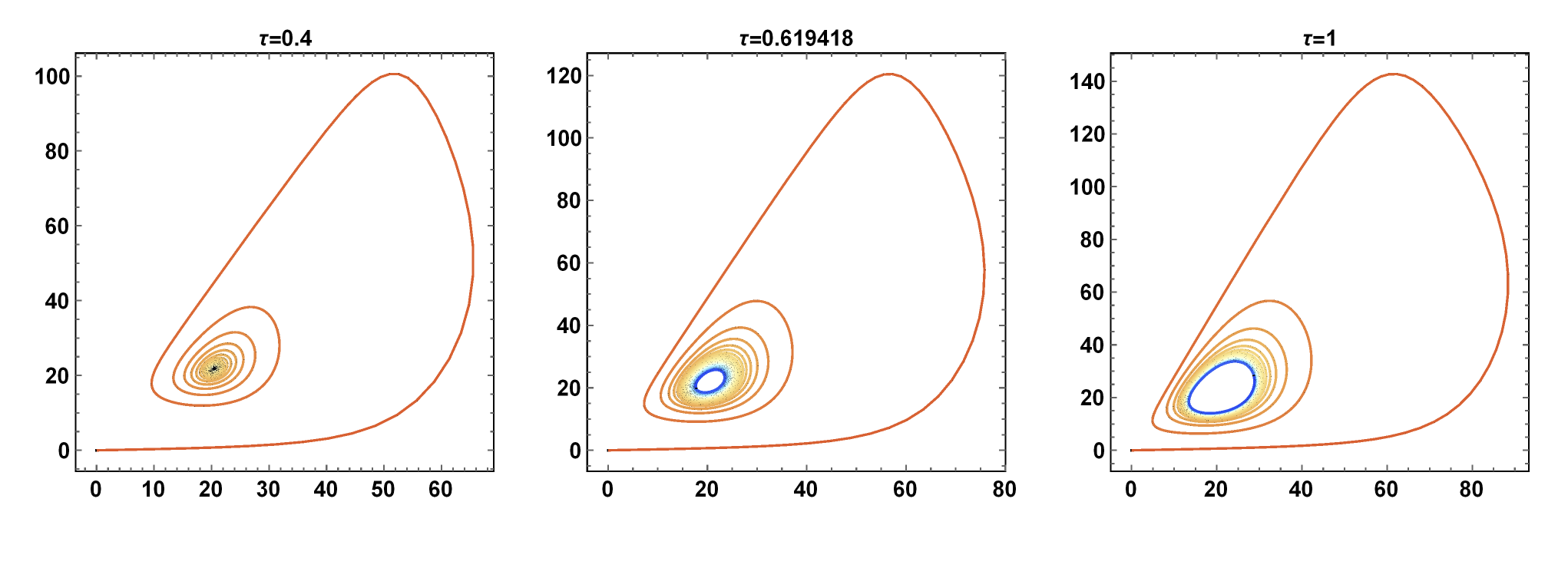}
	\includegraphics[width=\linewidth]{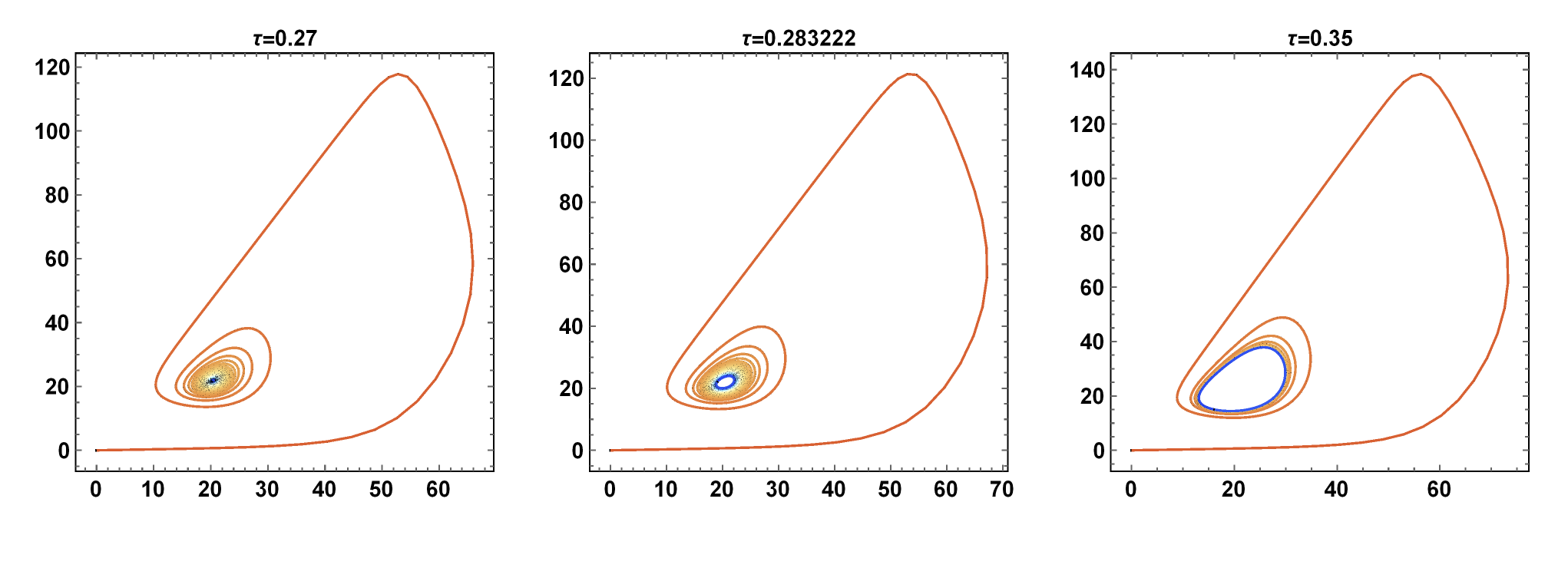}
	\caption{Phase plane trajectories for system \eqref{sys.holgado} with diseased state parameters, considering a Dirac kernel (first row, $\tau_{0,P}^*= 0.216411$), a weak Gamma kernel (second row, $\tau_{1,P}^*= 0.619418$) and a strong Gamma kernel (third row, $\tau_{2,P}^*=0.283222$), obtained for different values of $\tau=\Delta t/\tau_S$ around the Hopf bifurcation values.}
	\label{fig.sim.strong.gamma.park.pp}
\end{figure}

\section{Discussion}
\label{discussion}
We have accomplished a local stability and bifurcation analysis of a generalized version of the Wilson-Cowan model of excitatory and inhibitory interactions in localized neuronal populations, incorporating general distributed delays. Essential differences have been pointed out for different scenarios involving diverse delay kernels,  (ranging from simple periodic to quasi-periodic and aperiodic behaviors). This emphasizes the potentially crucial importance of the kernel choice to the dynamic behavior of the system. 

Differentiating between specific behaviors is very important, since they represent different physiological rhythms observed empirically in neuronal E/I circuits. 

Recall that the system variables $u(t)$ and $v(t)$ represent meanfield firing activities in the excitatory and respectively inhibitory Wilson Cowan coupled populations. Then a stable equilibrium corresponds to convergence of the two populations' firing to almost constant rates, while a stable cycle represents meanfield oscillations between higher and lower firing rates in the two populations.

Oscillatory behavior is ubiquitous in the brain, as a way for neurons to efficiently assemble in a synchronized fashion, optimal for receiving and sending information~\cite{buzsaki2006rhythms}. Oscillations are typically viewed as a result of balancing excitation and inhibition, with either gaining ground at different times within an oscillation cycle. 

Brain rhythms cover frequency bands from lower than 1Hz to 600 Hz, with different rhythms associated with different brain areas, functions and states. Moreover, this correspondence is not one-to-one, in the sense that two different brain circuits may generate oscillations in the same frequency band (likely by using different underlying mechanisms). In turn, the function of a particular oscillatory pattern largely depends on the underlying brain structure. For example, Beta rhythms (13-30 Hz) in the motor system are associated with the absence of movement.

Rhythms alterations in a brain circuit outside of its healthy range have been found by empirical studies to be associated with mental or neurological disorders (for example, exacerbated synchronization is the mark of epilepsy, Gamma oscillations have been found to be diminished in schizophrenia, and Parkinson's Disease has been tied to enhanced Beta oscillations in motor areas).  

However, the relationship between pathology and oscillations is rather complex, going beyond simply locating increased amplitude oscillations within a specific frequency band. Transition paths between healthy and pathological rhythms are still under investigation for most disorders, and recent research has been striving to define and understand them. Along these lines, our bifurcation analyses may reveal potential mechanisms that govern transitions between oscillatory behaviors in PD.

\color{black}

With this in mind, our results show that a weak Gamma kernel may be a preferable mode of input integration in populations whose function is based primarily on sustaining an approximately constant rate of firing. A strong Gamma kernel may be used as an input integration profile when the basic function of the circuit requires transitions into simple, regular { oscillations}. Finally, the Dirac kernel (which can mathematically be viewed as a subtler, limit-like profile for distributed delays) may be the signature of finer oscillators, whose function is modulated by transitions between more complex { firing} patterns. For each type of kernel, our results suggest that more specific behaviors, and transitions between dynamic regimes can be further controlled by the connectivity strengths, and by the delay $\tau$.

The importance of the type and length of delay to functional brain rhythms, supported by our theoretical results and numerical simulations, was further illustrated in an application to a Wislon-Cowan type  model with distributed delays of the basal ganglia, inspired by the work of Holgado et. al.\cite{holgado2010}, {{investigating control mechanisms of Beta rhythms in Parkinson's Disease.}

In the context of normal executive control, Beta oscillations have been associated with promoting existing movement, at the cost of initiating a new motor set~\cite{little2014functional}. In this context, it does not seem surprising that oscillations in the basal ganglia are significantly more pronounced in PD. Deep brain stimulation studies suggest in fact a nontrivial relationship, in which beta oscillatory behavior contributes both quantitatively (higher amplitudes) but also causally to the motor impairment in PD. Hence understanding the timing and conditions that favor onset of Beta activity is crucial, since understanding this causal relationship may hold potential for therapeutic interventions.}

\color{black}

The Holgado model { places these questions in the framework of a dynamical system formed of two coupled state variables, representing activation of} the STN and GP, as they respond to outside stimulation, and to modulations from other brain areas (the cortex and the striatum). Based on empirical evidence, the authors defined in the reference different ranges of pairwise synaptic coupling strengths between these brain areas, one characteristic to healthy brain functioning, the other corresponding to the dynamic patterns found in Parkinson's Disease. They found a prevalence for stable oscillations in the PD regime of the model, corresponding to physiological { rhythms} in the Beta-range in the corresponding brain areas of PD patients. We advanced the analysis of this model by investigating the importance of the coupling range to the system's behavior, {\it in conjunction} with the size and the distribution of the delays. We found that, for all kernels, stable oscillations are accessible to the system in the PD coupling range at shorter delays than to the system in the healthy range (exactly how short, depends on the type of kernel used). 

Since the shorter delays are compatible with physiological values, lower critical delays can be seen as a gate to promoting the Beta-range { oscillations} in PD patients. Comparably short delay values are significantly more likely to place a system with healthy synaptic coupling within the stability domain (as shown in Figure~\ref{freq}a). This corresponds to what is known about healthy basal ganglia function, which { promotes desynchronization of Beta oscillations during} executive responses. In contrast, in a system within the PD functioning range, { Beta oscillations persist even as the execution is being carried out. This may contribute to} the motor abnormalities found in PD (such as bradykinesia and rigidity, which have been empirically associated with enhanced Beta rhythms). Altogether, our analysis of Holgado's model of basal ganglia suggests that it may be important to investigate neural mechanisms that have the potential to reshape a cell's memory profile when processing inputs, as a possible means to { mitigate executive function abnormalities} in PD patients.

\begin{figure}
    \centering
    \includegraphics[width=\textwidth]{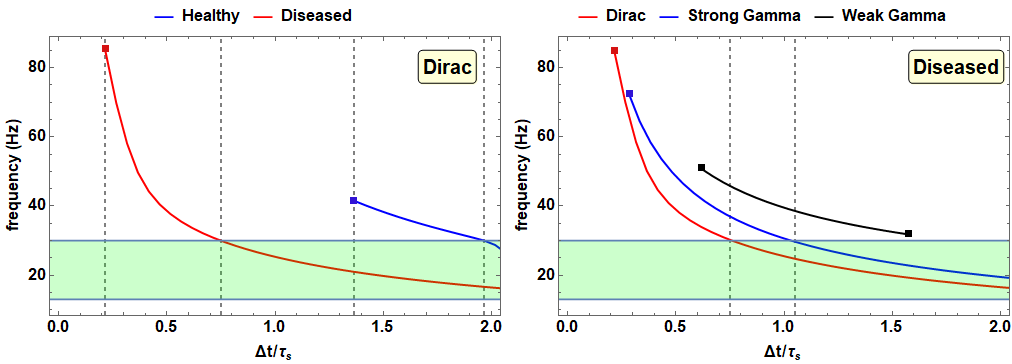}
    \caption{{Oscillation frequencies observed over an interval of values for $\Delta t/\tau_S$, in system \eqref{sys.holgado}. On each curve, the entry and exit form the oscillatory regime are marked with a dot. The Beta range is shown as a green shaded area. {\bf Left.} Comparison of oscillation frequencies between the system with a parameter set within the healthy range, versus one from the pathological range, when a Dirac kernel is considered. The diseased system enters the Beta range at a much lower $\Delta t/\tau_S \sim 0.74$ than the healthy system. {\bf Right.} Comparison of oscillation frequencies for the system with a parameter set within the pathological range, when different kernels are considered. The system with the Dirac kernel enters Beta oscillations at a lower value $\Delta t/\tau_S$ than the one with the strong Gamma kernel. The system with the weak Gamma kernel never exhibits Beta range frequencies within its window of oscillatory activity.}}
    \label{freq}
\end{figure}

We additionally observed particular patterns for the critical delay (onset of { oscillations}) { as well as for entering the Beta range activity, in the case of each kernel (Fingure~\ref{freq}b)}. We considered a system within the PD coupling regime, and with a distribution of delays within the physiological range, this combination placing it on a dynamic path of perpetual { oscillations}, even in absence of any additional stimulation. We explored the possibility that cessation of { oscillations} may be obtained by enhancing a target synaptic pathway. We noticed that, depending on the context, the perturbation may not necessarily have to be as large as to place the system in the healthy coupling regime. The adjustment would need to be just significant enough to move the system to a coupling profile where a shorter time delay is required in order to maintain { oscillatory behavior}; hence, with the current time delay, { oscillations} will be suppressed. Our future work aims to further study the potential physiological implementation of such mathematical solutions, and their potential towards development of clinical treatments. 

A significant limitation of the application in its current form is that brought by using the original, two-dimensional Wilson-Cowan coupling scheme. A two-dimensional system is a great place to start, since it is substantially easier to analyze and understand than higher dimensional, more complex systems. However, the neuronal circuit that governs executive function (and that has traditionally been studied in conjunction with conditions associated with executive { disturbances}, like Parkinson's Disease, or Obsessive Compulsive Disorder) crucially encompasses additional brain areas and connections, not represented in the Holgado model. For example, PD has been linked in both experimental and computational literature to modifications in the striatal-GP pathway~\cite{day2006selective,zhai2018striatal}, suggesting that the striatum should be included and studied as an additional state variable in the coupled system.

This underlines the general importance of extending the analysis of the Wilson-Cowan model with distributed delays to higher dimensional networks. We expect theoretical approaches to become more difficult, being affected by the fact that dynamic complexity increases vastly past dimension two, by the increased size of the coupling parameter space, and by the additional  contributions of the network architecture to the dynamics). However, the study of a more complete version of the importance of delays in the circuit controlling executive function could complement in novel and crucial ways existing work on modeling this circuit in absence of delays.


\bibliography{biblio}

\end{document}